\documentclass[]{interact}

\usepackage{amsthm,amsmath,amsfonts,amssymb,bm}

\usepackage[utf8]{inputenc}

\usepackage[table,dvipsnames]{xcolor}
\usepackage{hhline}
\usepackage[labelfont=bf, textfont = footnotesize]{caption}
\usepackage{stmaryrd}
\usepackage{caption}
\usepackage{enumerate}
\usepackage{hyperref}
\usepackage{bbm}
\usepackage{ushort}
\usepackage{dsfont}
\RequirePackage{booktabs}
\usepackage{comment}

\usepackage[numbers,sort&compress]{natbib}% Citation support using natbib.sty
\bibpunct[, ]{[}{]}{,}{n}{,}{,}% Citation support using natbib.sty
\renewcommand\bibfont{\fontsize{10}{12}\selectfont}% Bibliography support using natbib.sty
\makeatletter% @ becomes a letter
\def\NAT@def@citea{\def\@citea{\NAT@separator}}% Suppress spaces between citations using natbib.sty
\makeatother% @ becomes a symbol again

\usepackage{graphicx}

\newcommand{\reff}[1]{(\ref{#1})}

\newcommand{\IW}{\mathbb{W}}

\newcommand{\IL}{\mathbb{L}}

\newcommand{\sH}{\mathcal{H}}
\newcommand{\sG}{\mathcal{G}}

\newcommand{\sP}{\mathcal{P}}

\newcommand{\sW}{\mathcal{W}}

\newcommand{\sM}{\mathcal{M}}

\newcommand{\IR}{\mathbb{R}}
\newcommand{\IRp}{\mathbb{R}_+}
\newcommand{\IC}{\mathbb{C}}
\newcommand{\IZ}{\mathbb{Z}}
\newcommand{\IN}{\mathbb{N}}
\newcommand{\IE}{\mathbb{E}}
\newcommand{\E}{\mathbb{E}}
\newcommand{\IP}{\mathbb{P}}

\newcommand{\Ii}{\mathbbm{1}}

\newcommand{\Var}{\mathbb V\mathrm{ar}}

\newcommand{\nset}[1]{\llbracket #1 \rrbracket}

\DeclareMathOperator*{\argmin}{\arg\min}

\newtheorem{thm}{Theorem}[section]
\newtheorem{lem}[thm]{Lemma}
\newtheorem{cor}[thm]{Corollary}
\newtheorem{prop}[thm]{Proposition}

\theoremstyle{definition}

\newtheorem{rem}[thm]{Remark}

\newtheorem{exs}[thm]{Examples}

\numberwithin{equation}{section}

\begin{document}

%\articletype{Statistics}

\title{Volatility density estimation by multiplicative deconvolution}
%\title{A sample article title with some additional note\thanksref{T1}}
%\runtitle{Multiplicative deconvolution in survival analysis under dependency}
%\thankstext{T1}{A sample of additional note to the title.}

\author{
\name{Sergio Brenner Miguel\textsuperscript{a}}
\affil{\textsuperscript{a}Institut f\"{u}r angewandte Mathematik und Interdisciplinary Center for Scientific Computing (IWR), Im Neuenheimer Feld 205, Heidelberg University, Germany}
}

\maketitle

\begin{abstract}
	We study the non-parametric estimation of an unknown stationary density fV of an unobserved strictly stationary volatility process $(\bm V_t)_{t\geq 0}$ on $\IRp^2 := (0,\infty)^2$ based on discrete-time observations in a stochastic volatility model. We identify the under- lying multiplicative measurement error model and build an estimator based on the estimation of the Mellin transform of the scaled, integrated volatility process and a spectral cut-off regularisation of the inverse of the Mellin transform. We prove that the proposed estimator leads to a consistent estimation strategy. A fully data-driven choice of $\bm k \in \IRp^2$ is proposed and upper bounds for the mean integrated squared risk are provided. Throughout our study, regularity properties of the volatility process are necessary for the analsysis of the estimator. These assumptions are fulfilled by several examples of volatility processes which are listed and used in a simulation study to illustrate a reasonable behaviour of the proposed estimator.
\end{abstract}


\section{Introduction}\label{sec_intro}
In this work, we are interested in estimating the unknown stationary density $f_{\bm V}: \IRp^2\rightarrow \IRp$ of an unobserved, strictly stationary volatility process $(\bm V_t)_{t\geq 0}, \bm V_t=(V_{t,1}, V_{t,2})^T$ in a stochastic volatility model with discrete-time observations. More precisely, we assume that we have access to the discrete-time observations $\bm Z_{\Delta}, \dots,  \bm Z_{\Delta n}, n\in \mathbb N, \Delta \in (0,1)$, of the solution $(\bm Z_t)_{t\geq 0}$ of the stochastic differential equation
\begin{align}\label{st:vol:mult}
d\bm Z_t = \bm \Sigma_t d\bm W_t, \quad \bm \Sigma_{ t}:= \begin{pmatrix} \sqrt{V_{t,1}} & 0 \\ 0 & \sqrt{V_{t,2}} ,\end{pmatrix}  \quad \bm Z_0:=\begin{pmatrix} 0 \\ 0 \end{pmatrix},
\end{align}
where $(\bm W_t)_{t\geq 0}, \bm W_t=(W_{t,1}, W_{t,2})^T$ is a standard Brownian motion on $\IR^2$, stochastically independent of $(\bm V_t)_{t\geq 0}.$ \\
In the non-parametric literature, the stochastic volatility model has been intensively studied in the earlier 2000s. Introduced by \cite{HullWhite1987} as a natural expansion of the constant volatility model studied by \cite{BlackScholes1973}, the interpretation of the volatility as a  stochastic process itself enabled the theory to explain in-practice-observed phenomenons, as pointed out by \cite{RenaultTouzi1996}. \\
The stochastic volatility model has been intensively studied by the authors of \cite{Genon-CatalotJeantheauLaredo1998}, \cite{Genon-CatalotJeantheauLaredo1999} and \cite{Genon-CatalotJeantheauLaredo2000} developing limit theorems of the empirical distribution, studying parameter estimation and including the model in a hidden markov model framework.\\
Later on, non-parametric estimators have been studied for instance by \cite{Comte2004} and \cite{Van-EsSpreijVan-Zanten2003} where \cite{Comte2004} considered a regression-type estimation problem while \cite{Van-EsSpreijVan-Zanten2003} considered the point-wise estimation of the stationary density of the volatility process. Both, \cite{Van-EsSpreijVan-Zanten2003} and \cite{Comte2004} studied kernel estimators and univariate volatility processes. The generalisation of \cite{Van-EsSpreijVan-Zanten2003} for multivariate volatility processes was done by \cite{Van-EsSpreij2011} with an isotropic choice of the bandwidth, while a different structure of multivariate volatility processes had been considered in \cite{Van-EsSpreijVan-Zanten2003}. A penalised projection estimator of the stationary density was studied  in \cite{ComteGENON-CATALOT2006}. Assuming that the volatility process is an diffusion process \cite{ComteGenon-CatalotRozenholc2010} proposed a penalised projection estimator for the volatility and drift coefficients in a stochastic volatility model. \\
Frequently, the mentioned authors built their non-parameteric estimators on a $\log$-transformation of the data in order to rewrite the estimation problem into an addditive deconvolution problem and use standard deconvolution estimators. This was a common strategy in the non-parametric literature to adress multiplicative errors. In contrary to this strategy, \cite{BelomestnyGoldenshluger2020} studied the mutliplicative measurement error model directly by using the Mellin transform to solve the underlying multiplicative convolution.  \cite{BelomestnyGoldenshluger2020} proposed a kernel density estimator and studied its pointwise risk. Based on this work, \cite{Brenner-MiguelComteJohannes2021} constructed a spectral cut-off estimator in the multiplicative measurement error model with global risk. \cite{Brenner-Miguel2021} then generalised the results of \cite{Brenner-MiguelComteJohannes2021}, which are stated for univariate variables, for multivariate density estimation under multiplicative measurement errors with anisotropic choice of the smoothing parameter.\\
Based on the results of \cite{Brenner-Miguel2021}, we will consider a multivariate stochastic volatility model, similar to the one considered in \cite{Van-EsSpreij2011}, and propose an anisotropic non-parametric estimator of the stationary density exploiting the rich theory of Mellin transforms.\\
Our approach differs in the following way from the existing literature. Instead of using a $\log$-transformation of the data, we adress the multiplicative deconvolution problem directly. Despite the fact that this seems to be more natural, we are additionally able to identify and study the underlying inverse problem in a more convenient way, as done in \cite{Brenner-Miguel2021} and state more general results. Indeed, our results  include the $\log$ transformation approach as a special case, as pointed out by \cite{BelomestnyGoldenshluger2020} and \cite{Brenner-MiguelComteJohannes2021}. In contrary to \cite{Van-EsSpreij2011}, we study an anisotropic choice of the smoothing parameter which in general leads to a more flexible estimator, compare \cite{ComteLacour2013} and \cite{Brenner-Miguel2021}.\\
The paper is structured as follows. In Section \ref{sec_intro}, we introduce the bivariate stochastic volatility model, identify the underlying multiplicative deconvolution problem and collect the regularity assumptions on the volatility process $(\bm V_t)_{t\geq 0}$. In Section \ref{sec_svde}, we introduce the Mellin transform and build an estimator based on the estimation of the Mellin transform of the scaled, integrated volatility process and a spectral cut-off regularisation of the inverse Mellin transform. We measure the performance of our estimator in terms of the mean integrated squared error and provide upper bounds for arbitrary choices of $\bm k\in \IRp$. We then propose a fully data-driven choice of $\bm k\in \IRp^2$, based only on the observations $\bm Z_{\Delta}, \dots, \bm Z_{n\Delta}$ and bound the risk of the resulting data-driven density estimator. Several examples of volatility processes are then studied in Section \ref{sec_numerical} and used in a simulation study to show reasonable the performance of the proposed estimation strategy. More general results for the density estimation in a multiplicative measurement error model with stationary data are stated in Section \ref{sec_stat_pro}, which are needed in the proofs of the results of Section \ref{sec_svde}. The proof ofs Section \ref{sec_intro}, \ref{sec_svde} and \ref{sec_stat_pro} are collected in the Appendix \ref{sec_append}.

\paragraph*{Stochastic volatility model} In this paper, we consider the following version of a multivariate stochastic volatility model, motivated by \cite{Danielsson-Jon1998}, which has also been considered by \cite{Van-EsSpreij2011}.\\
For a strictly stationary unobserved Markov process $(\bm V_t)_{t\geq 0}$, we consider the solution $(\bm Z_t)_{t\geq 0}$ of the stochastic differential equation \eqref{st:vol:mult}
where $(\bm W_t)_{t\geq 0}$ is a standard $2$-dimensional Brownian motion, stochastically independent of the process $(\bm V_t)_{t\geq 0}$. Then motivated by the work \cite{Genon-CatalotJeantheauLaredo2003}, respectively \cite{ComteGENON-CATALOT2006}, we study the scaled increments of our discrete-time sample $(\bm Z_{\Delta j})_{j\in \nset{n}}$ for $\Delta\in (0,1)$ and $\nset{n}:=[1, n]\cap \mathbb N$. \\
More precisely, let $\bm D_j:=\Delta^{-1/2}(\bm Z_{\Delta j}- \bm Z_{\Delta(j-1)})$, understood componentwise for $j\in\nset{n}$ , then conditioned on $(\bm V_t)_{t\geq 0}$ we have
$$ \bm D_j=  \frac{1}{\sqrt{\Delta}}\begin{pmatrix} \int_{(j-1)\Delta}^{j\Delta} V_{t,1} d W_{t,1}\\
\int_{(j-1)\Delta}^{j\Delta} V_{t,2} d W_{t,2}\end{pmatrix} \sim \mathrm{N}_{(\bm 0, \bm \Sigma_{\bm V_j})},\quad \bm \Sigma_{\bm V_j}:= \begin{pmatrix} \overline{V}_{j,1} &0 \\ 0 &  \overline{V}_{j,2}\end{pmatrix}$$
exploiting the independence of $(\bm V_t)_{t\geq 0}$ and $(\bm W_t)_{t\geq 0}$, where $\overline V_{j,\ell}:=\Delta^{-1}\int_{(j-1)\Delta}^{j\Delta} V_{s, \ell} ds$, $\ell \in \{1,2\}$. As a direct consequence, we write
\begin{align}\label{eq:model:gen}
\bm Y_j:=\begin{pmatrix} Y_{j,1} \\ Y_{j,2} \end{pmatrix} := \begin{pmatrix} D^2_{j,1} \\ D^2_{j,2} \end{pmatrix}=\begin{pmatrix} \overline V_{j, 1} U_{j,1}\\ \overline V_{j, 2} U_{j,2} \end{pmatrix}=:\begin{pmatrix} X_{j,1} U_{j,1} \\ X_{j,2} U_{j,2} \end{pmatrix}  =: \bm X_j \bm U_j 
\end{align}
where $\bm X_j$ and $\bm U_j$ a stochastically independent and $(\bm U_j)_{j\in \nset{n}}$ is an i.i.d. (independent, identically distributed) sequence with $ U_{1,1}, U_{1,2}\stackrel{i.i.d.}{\sim} \chi^2_1=\Gamma_{(1/2,1/2)}$. In other words, the stochastic volatility model can be expressed as a multiplicative measurement error model with $\chi$-squared, respectively Gamma distributed noise. While the authors from \cite{Genon-CatalotJeantheauLaredo2003}, \cite{ComteGENON-CATALOT2006}, \cite{Van-EsSpreijVan-Zanten2003} and \cite{Van-EsSpreij2011} used a $\log$-transformation of the data, we will instead exploit the theory of multivariate Mellin transform and their use in non-parametric density estimation introduced in \cite{Brenner-Miguel2021} to build a multiplicative deconvolution density estimator.
\paragraph*{Assumption on the volatility process $(\bm V_t)_{t\geq 0}$}
Throughout this paper, we will need to assume some regularity of the volatility process $(\bm V_t)_{t\geq 0}$ to ensure the well-definedness of the upcoming objects and to deduce consistency of our proposed estimation strategy. As usual in non-parametric approaches, we aim to consider an ensemble of assumptions which can be proven for a wide class of examples of volatility processes. To motivate that these assumptions are not restrictive, we will show in Section \ref{sec_numerical} a number of examples of frequently studied volatility processes.\\
 Now let us assume that the discrete-time sample $(\bm Z_{j\Delta})_{j\in \nset{n}}$ is drawn from a process $(\bm Z_t)_{t\geq 0}$ solving \eqref{st:vol:mult} where
\begin{enumerate}
	\item[($\bm{\mathrm{A}_0}$)] $(\bm W_t)_{t\geq 0}$ is a two-dimensional Brownian motion, independent of the process $(\bm V_t)_{t\geq 0}$ on $\IRp^2$,
	\item[($\bm{\mathrm{A}_1}$)] $(\bm V_t)_{t\geq 0}$ is a time-homogeneous Markov process, with continuous sample paths, strictly stationary and ergodic. The stationary distribution of $(\bm V_t)_{t\geq 0}$ admits a density $f_{\bm V}$ with respect to the Lebesgue measure on $\IRp^2$,
	\item[$(\bm{\mathrm{A}_2})$] $(\bm V_t)_{t\geq 0}$ is $\beta$-mixing, with $\int_{\IRp} \beta_{\bm V}(s)ds <\infty$, where $$\beta_{\bm V}(s)=\mathrm{TV}(\IP^{(\bm V_0, \bm V_s)}, \IP^{\bm V_0}\otimes \IP^{\bm V_s}), \quad s\in \IRp,$$ where $\mathrm{TV}$ is the total variation distance.
\end{enumerate}
For the estimation we will be in need of the following additional assumption
\begin{enumerate}
	\item[$(\bm{\mathrm A_3})$] There exists a constant $\mathfrak c>0$ such that $\E(|\log(X_{1,1})-\log(V_{0,1})|+|\log(X_{1,2})-\log(V_{0,2})|) \leq \mathfrak c\sqrt{\Delta}$.
\end{enumerate}

In Section \ref{sec_numerical}, we will deliver examples of volatility processes $(\bm V_t)_{t\geq 0}$ which satisfy $(\bm{\mathrm A_0})$-$(\bm{\mathrm A_3})$. While assumptions $(\bm{\mathrm A_0})$- $(\bm{\mathrm A_2})$ are widely considered in the literature and proven for several diffusion processes, assumption $(\bm{\mathrm A_3})$ is of rather technical nature. A practical proposition in the univariate case was proposed by \cite{ComteGENON-CATALOT2006}. Here, we want to state a bivariate counterpart. The proof of Proposition \ref{prop:ass:3} can be found in Section \ref{a:intro}. Here, we denote for $\bm a\in \IR^2$ the Euclidean norm by $|\bm a|_{\IR^2}^2:=a_1^2+a_2^2$ and for a matrix $\bm A \in \IR^{(2,2)}$ the Frobenius norm by $|\bm A|_{F}^2:=A_{1,1}^2+A_{1,2}^2+A_{2,1}^2+A_{2,2}^2.$ Furthermore, for $p\in \mathbb N_0$ we denote by $\mathcal C^p(\mathcal D)$ the set of all $p$-times continuously differentiable functions on $\mathcal D\subseteq\IR^2$
\begin{prop}\label{prop:ass:3}
	Suppose the volatility process $(\bm V_t)_{t\geq 0}$ satisfies (either) one of the following conditions 
	\begin{enumerate} 
		\item $\bm V_t = (\exp(Z_{t,1}), \exp(Z_{t,2}))^T, t\geq 0,$ where $(\bm Z_t)_{t\geq 0}$ is a strictly stationary and ergodic diffusion process on $\IR^2$ satisfying $d\bm Z_t= \widetilde{\bm b}(\bm Z_t)+ \widetilde{\bm a}(\bm Z_t)d \widetilde{\bm W}_t$, $(\widetilde{\bm W}_t)_{t\geq 0}$ standard Brownian motion on $\IR^2$ such that there exists $\widetilde L>0$ with $$\|\widetilde{\bm a}(\bm x)\|_F +|\widetilde{\bm b}(\bm x)|_{\IR^2} \leq \widetilde{L} (1+|\bm x|_{\IR^2})$$
		for all $\bm x\in \IR^2$, $\widetilde{\bm b}_i, \widetilde{\bm a}_{i,j} \in \mathcal C^0(\IR^2)\cap C^1(\IRp^2)$ for $i,j\in \llbracket 2 \rrbracket$ and $\E(|\bm Z_{0}|_{\IR^2}^2)< \infty$ or
		\item or $(\bm V_t)_{t\geq 0}$ is a strictly stationary and ergodic diffusion process on $\IRp^2$ satisfying $d\bm V_t=\bm b(\bm V_t)+\bm a(\bm V_t) d\widetilde{\bm W}_t$ such that there exists $L>0$ with
		$$\|{\bm a}(\bm x)\|_F +|{\bm b}(\bm x)|_{\IR^2} \leq \widetilde{L} (1+|\bm x|_{\IR^2})$$
		for all $\bm x\in \mathbb R^2$, $\bm b_i, \bm a_{i,j}\in \mathcal C^0(\IR^2)\cap C^1(\IRp^2)$ for $i,j \in \nset{2}$ and additionally let $\E(\sup_{t\in [0,\Delta]} (V_{t,i})^{-2} ), \E(|\bm V_0|_{\IR^2}^2)<\infty$ for $i\in \llbracket 2\rrbracket$ hold true.
	\end{enumerate}
	Then $(\bm V_t)_{t\geq 0}$ satisfies $(\bm{\mathrm A_3})$.
\end{prop}
After this brief introduction to the stochastic volatility model, let us propose a non-parametric density estimator based on a multiplicative deconvolution.

\section{Stochastic volatility density estimation}\label{sec_svde}
In this section we introduce the Mellin transform and start to collect some of its major properties, which are stated in \cite{Brenner-Miguel2021}. We then propose our estimator.
\paragraph*{Notations and definitions of the Mellin transform}
For two vectors $\bm u=(u_1,u_2)^T, \bm v=(v_1,v_2)^T \in \IR^2$ and a scalar $\lambda \in \IR$ we define the componentwise multiplication  $\bm u\bm v := \bm u\cdot \bm v:=(u_1v_1, u_2v_2)^T$ and denote by $\lambda \bm u$ the usual scalar multiplication. Further, if $v_1, v_2\neq 0$ we define the multivariate power by $\bm v^{\ushort{\bm u}}:= v_1^{u_1}v_2^{v_2}$. Additionally, we define the componentwise division by $ \bm u /\bm v := (u_1/v_1, u_2/v_2)^T$ We denote the usual Euclidean scalar product and norm on $\IR^2$ by $\langle \bm u, \bm v\rangle_{\IR^2} := \sum_{i\in \nset{2}} u_i v_i$ and $|\bm u|_{\IR^2}:= \sqrt{\langle \bm u, \bm u \rangle_{\IR^2} }$. Moreover, we set $\bm 1:=(1,1)^T\in \IRp$, respectively $\bm 0:=(0,0)^T.$\\
For a positive random vector $\bm Z$ with $\E(\bm Z^{\ushort{\bm c-\bm 1}})=\E(Z_1^{c_1-1}Z_2^{c_2-1})<\infty$, $\bm c \in \IR^2$, we define the Mellin transform $\sM_{\bm c}[\bm Z]$ of $\bm Z$ as the function $$\sM_{\bm c}[\bm Z]: \IR^2\rightarrow \IC,\quad  \bm t\mapsto \sM_{\bm c}[\bm Z](\bm t):= \E(\bm Z^{\ushort{\bm c-\bm 1+i\bm t}}).$$
As a consequence the convolution theorem for the Mellin transform holds true, that is
for $\bm U \bm V$ independent with $\E((\bm U\bm V)^{\ushort{\bm c-\bm 1}})< \infty$, 
$$\sM_{\bm c}[\bm U\bm V](\bm t)=\sM_{\bm c}[\bm U](\bm t) \sM_{\bm c}[\bm V](\bm t),\quad \bm t\in \IR^2.$$ If $\bm Z$ emits a Lebesgue density $h:\IRp^2\rightarrow \IRp$, then we can write $\sM_{\bm c}[\bm Z](\bm t)= \int_{\IR_+^2} \bm x^{\ushort{\bm c-\bm 1+i\bm t}} h(\bm x) d\bm x$, $t\in \IR^2$. Motivated by this, we define the set $\IL^1(\IR_+^2,\bm x^{\ushort{\bm c-\bm 1}}):=\{h: \IR_+^2  \rightarrow \IC: \|h\|_{\IL^1(\IR_+^2,\bm x^{\ushort{\bm c-\bm 1}})}:= \int_{\IR^2_+} |h(\bm x)|\bm x^{\ushort{\bm c-\bm 1}}d\bm x < \infty \}$. Then we can generalise the notion of the Mellin transform for $\IL^1(\IR_+^2,\bm x^{\ushort{\bm c-\bm 1}})$ function. Indeed, for $h\in \IL^1(\IR_+^2,\bm x^{\ushort{\bm c-\bm 1}})$ we define the \textit{Mellin transform} of $h$ at the development point $\bm c\in \mathbb R^2$ as the function $\mathcal M_{\bm c}[h]:\mathbb R^2\rightarrow \mathbb C$ by
\begin{align}
\mathcal M_{\bm c}[h](\bm t):= \int_{\IRp^2} \bm x^{\ushort{\bm c-\bm 1+i\bm t}} h(\bm x)d\bm x, \quad \bm t\in \mathbb R^2.
\end{align}
In analogy to the Fourier transform, one can define the Mellin transform for square integrable functions. We
define the weighted norm by $\|h\|_{\bm x^{\ushort{2\bm c-\bm 1}}}^2 := \int_{\IR_+^2}
|h(\bm x)|^2\bm x^{\ushort{2\bm c-\bm 1}}d\bm x $ for a measurable function $h:\IR_+^2 \rightarrow \IC$ and denote by
$\IL^2(\IR_+^2,\bm x^{\ushort{2\bm c-\bm 1}})$ the set of all complex-valued, measurable functions with
finite $\|\, .\,\|_{\bm x^{\ushort{2\bm c-\bm 1}}}$-norm and by $\langle h_1, h_2
\rangle_{\bm x^{\ushort{2\bm c-\bm 1}}} := \int_{\IRp^2}  h_1(\bm x)
h_2(\bm x)\bm x^{\ushort{2\bm c-\bm 1}}d\bm x$ for $h_1, h_2\in \IL^2(\IRp^2,\bm x^{\ushort{2\bm c-\bm 1}})$
the corresponding weighted scalar product. Similarly, we define $\IL^2(\IR^2):=\{ H:\IR^2 \rightarrow \IC\, \text{ measurable }: \|H\|_{\IR^2}^2:=  \int_{\IR^2} H(\bm x)\overline{H(\bm x)} d\bm x <\infty \}$. \\
We are then able to define the Mellin transform as the isomorphism $\sM_{\bm c}: \IL^2(\IRp^2, \bm x^{\ushort{2\bm c-\bm1}})\rightarrow \IL^2(\IR^2).$ For a precise definition of the multivariate Mellin transform and its connection to the Fourier transform, we refer to \cite{Brenner-Miguel2021}. Nevertheless, if $h\in \IL^1(\IRp^2,\bm x^{\ushort{\bm c-\bm 1}}) \cap \IL^2(\IRp^2,\bm x^{\ushort{2\bm c-\bm 1}})$ both notions coincide. By abuse of notation we will denote by $\sM_{\bm c}[h]$ both notions, for $h\in   \IL^1(\IRp^2,\bm x^{\ushort{\bm c-\bm 1}})$, respectively $h\in \IL^2(\IRp^2,\bm x^{\ushort{2\bm c-\bm 1}})$ of the Mellin transform. For a more detailed collection of the properties of the Mellin transform we refer to Section \ref{sec_stat_pro}, respectively \cite{Brenner-Miguel2021}.
\paragraph*{Estimation strategy}
For $\bm k\in \IRp^2$ we define the hypercuboid $[-\bm k, \bm k]:=[-k_1, k_1]\times [-k_2, k_2]$.
Then based on the work of \cite{Brenner-Miguel2021}, we define for any ${\bm c}\in \IR^2$ with $\E(\bm Y_1^{\ushort{{\bm c}-\bm 1}})< \infty$ and $\bm k\in \IRp^2, \Delta \in (0,1)$ the estimator $\widehat f_{\Delta, \bm k}$ by  
\begin{align}\label{def:est}
\widehat f_{\Delta,\bm k}(\bm x):= \frac{1}{4\pi^2} \int_{[-\bm k, \bm k]} \bm x^{\ushort{-\bm {{\bm c}}-i\bm t}} \frac{\widehat{\mathcal M}_{{\bm{c}}}(\bm t)}{\mathcal M_{{\bm{c}}}[g](\bm t)} d\bm t , \quad \bm x, \bm k\in \IRp^2,
\end{align} 
where $g:\IR_+^2\rightarrow \IR_+$ is the density of $\bm U_1$ and $\widehat{\sM}_{\bm c}(\bm t):=n^{-1} \sum_{j\in \nset{n}} \bm Y_j^{\ushort{\bm c-\bm 1+i\bm t}}$, $\bm t\in \IR^2$, is the empirical Mellin transform of the sample $(\bm Y_{j})_{j\in \nset{n}}$. 
Here, the moment assumption is trivally fulfilled  for the case $\bm c=\bm 1$. We will mainly focus on this special case in this section while theoretical results for general choices of $\bm c\in \IR$ are given in Section \ref{sec_stat_pro}. \\Let us assume that $f_{\bm V}\in \IL^2(\IR_+^2, \bm x^{\ushort{\bm 1}})$. By construction \eqref{def:est}, we have $\widehat f_{\Delta, \bm k} \in  \IL^2(\IR_+^2, \bm x^{\ushort{\bm 1}})$ for any $\bm k\in \IR_+^2$. Furthermore, we define the approximation $f_{\bm V, \bm k}\in \IL^2(\IR_+^2, \bm x^{\ushort{\bm 1}})$ by
$$ f_{\bm V, \bm k}(\bm x):= \frac{1}{4\pi^2} \int_{[-\bm k, \bm k]} \bm x^{\ushort{-\bm 1-i\bm t}} \sM_{\bm 1}[f_{\bm V}](\bm t) d\bm t, \quad \bm x, \bm k \in \IR^2_+.$$ 
We can now show the following risk bound for the family of estimators $(\widehat f_{\Delta,\bm k})_{\bm k \in \IRp}$ presented in \eqref{def:est}, implying that for a suitable choice of the cut-off parameter $\bm k\in \IRp^2$ a consistent estimator can be achieved.
\begin{thm}[Upper bound of the risk]\label{pro:volatility}
	Let $f_{\bm V} \in \IL^2(\IRp^2, \bm x^{\ushort{\bm 1}})$ and assumptions $(\bm{\mathrm{A}_0})- (\bm{\mathrm{A}_3})$ hold true. Then, for any $\bm k\in \IRp^2$ and $\Delta\in (0,1)$,
	\begin{align*} 
	\E_{f_{\bm Y}}^n(\|f_{\bm V}-\widehat f_{\Delta,\bm k}\|_{\bm x^{\ushort{\bm 1}}}^2 ) 
	\leq &\|f_{\bm V}- f_{\bm V,\bm k}\|_{\bm x^{\ushort{\bm 1}}}^2 +\mathfrak c \Delta k_1^3k_2^3 +  \frac{1}{4\pi^2n }\int_{[-\bm k, \bm k]} |\sM_{\bm 1}[g](\bm t)|^{-2} d\bm t\\
	&+ \frac{k_1k_2}{n\Delta}\int_{\IRp} \beta_{\bm V}(s)ds
	\end{align*}
	where $\mathfrak{c}$ is defined in assumption $(\bm{\mathrm{A}_3})$. 
\end{thm}
While the squared bias term $\|f_{\bm V}-f_{\bm V,\bm k}\|_{\bm x^{\ushort{\bm 1}}}^2$ and $(4\pi^2n)^{-1} \int_{[-\bm k, \bm k]}|\sM_{\bm 1}[g](\bm t)|^{-2} d\bm t$ already arise in  \cite{Brenner-Miguel2021} in the multiplicative deconvolution setting for i.i.d. observations, the remaining two summands in the upper bound of Theorem \ref{pro:volatility} are specific to the stochastic volatility model. \\
More precisely, the last summand $k_1k_2 n^{-1} \int_{\IRp}\beta_{\bm V}(s)ds$ is an addtional variance part due to the underlying dependency of the observations $(\bm X_j)_{j\in\nset{n}}$, compare Proposition \ref{pr:consis}, respectively \cite{Brenner-MiguelPhandoidaen2022} for a similar arising term in context of survival function estimation under dependency. The second summand, $\mathfrak c\Delta k_1^3k_2^3$, on the other hand, is an additional bias term due to the fact, that the distributions of $\bm X_1$ and $\bm V_0$ differ. \\
It is interesting here that the additional bias term is decreasing for smaller values of $\Delta$ while the additional variance term is increasing for fixed values of $n\in \mathbb N$. The latter effect is natural, since for fixed $n\in \IN$, the time interval $[0, n\Delta]$, where we discretely derive our observations from, is vanishing. Therefore, a choice of $\Delta$ with respect to $n\in \IN$ is non-trivial. We will now focus on the variance term $(4\pi^2n)^{-1} \int_{[-\bm k, \bm k]}|\sM_{\bm 1}[g](\bm t)|^{-2} d\bm t$.\\
For $(2\pi)^{-2} \int_{[-\bm k, \bm k]} |\mathcal M_{\bm 1}[g](\bm t)|^{-2}d\bm t$, in the stochastic volatility model, we have $\mathcal M_{\bm1}[g](\bm t)=2^{i(t_1+t_2)}\pi^{-1}\Gamma(1/2+it_1)\Gamma(1/2+it_2)$ leading to
$$ |\mathcal M_{\bm 1}[g](\bm t)|^{-2}=\frac{\mathrm{cosh}^2(\pi t_1)\mathrm{cosh}^2(\pi t_2)}{(2\pi)^2}, \quad \text{where }\, \mathrm{cosh}(t):= \frac{\exp(t)+\exp(-t)}{2}$$
using the multiplication theorem of the $\Gamma$-function. This is an example of super smooth error densities considered for instance in \cite{BelomestnyGoldenshluger2020} and \cite{Brenner-MiguelComteJohannes2021a}. This implies the following corollary whose proof is omitted.\begin{cor}\label{cor:cons}
	Let $f_{\bm V}\in \IL^2(\IRp^2, \bm x^{\ushort{\bm 1}})$ and assumptions $(\bm{\mathrm{A}_0})- (\bm{\mathrm{A}_3})$ hold true. Then, for any $\bm k\in \IRp^2$, 
	\begin{align*} 
	\E_{f_{\bm Y}}^n(\|f_{\bm V}-\widehat f_{\Delta,\bm k}\|_{\bm x^{\ushort{\bm 1}}}^2 ) 
	&\leq \|f_{\bm V}- f_{\bm V,\bm k}\|_{\bm x^{\ushort{\bm 1}}}^2 +\mathfrak c \Delta k_1^3k_2^3 + \frac{e^{\pi(k_1+k_2)}}{n}+ \frac{k_1k_2}{n\Delta}\int_{\IRp} \beta_{\bm V}(s)ds
	\end{align*}
	where $\mathfrak{c}$ is defined in assumption $(\bm{\mathrm{A}_3})$. Now for any $\Delta=\Delta_n \rightarrow 0$ with $n\Delta_n\rightarrow \infty$ as $n\rightarrow \infty$ we can find a sequence $(\bm k_{n})_{n\in \IN}$ with $\bm k_n \rightarrow \bm\infty$, such that
	$$ \IE_{f_{\bm Y}}^n(\|f_{\bm V}-\widehat f_{\Delta,\bm k_n}\|_{\bm x^{\ushort{1}}}^2) \rightarrow 0,$$
	implying that $\|f_{\bm V}-\widehat f_{\Delta,\bm k_n}\|_{\bm x^{\ushort{1}}}^2\rightarrow 0$ in probability.
\end{cor} 
Although, Corollary \ref{cor:cons} implies the existence of $(\bm k_n, \Delta_n)_{n\in \mathbb N}$ such that $\widehat f_{\bm k_n, \Delta_n}$ is a consistent estimator of $f_{\bm V}$, a choice of $\bm k_n\in \IRp^2$ which minimises the risk would still depend on the decay of the squared bias term $\|f_{\bm V}-f_{\bm V, \bm k}\|_{\bm x^{\ushort{\bm 1}}}^2$ which, without further assumptions, is unknown. We therefore propose in the next paragraph a fully data-driven estimator based on the model selection approach presented in  \cite{Brenner-Miguel2021} with small adjustments inspired by the work of \cite{ComteLacour2013}. In Section \ref{sec_numerical}, we will study examples of volatility processes and deduce their expected rate.
\paragraph*{Data-driven choice of $\bm k\in \IRp^2$.}
First we reduce the space of possible cut-off parameters to $\mathcal K_n:=\{\bm k \in \nset{\log(n)} \times \nset{\log(n)} : \exp(\pi(k_1+k_2)) \leq n\}$. This reduction is rather natural, since any choice of $\bm k_n$, leading to a consistent estimator, implies that $\exp(\pi(k_{n,1}+k_{n,2}))n^{-1}$ goes to $0$ for $n\rightarrow \infty.$ We define the model selection method for $\chi>0$ by 
\begin{align}\label{eq:pco:sv}
\widehat{\bm k}:=\argmin_{\bm k\in \mathcal K_n} -\|\widehat f_{\Delta,\bm k}\|_{\bm x^{\ushort{\bm 1}}}^2 +\mathrm{pen}(\bm k), \quad\mathrm{pen}(\bm k):= \chi k_1k_2\exp(\pi(k_1+k_2))n^{-1}.
\end{align}
In comparison to the model selection in \cite{Brenner-Miguel2021}, the penalty term $(\mathrm{pen}(\bm k))_{\bm k\in \mathcal K_n}$ overestimates the variance. This is a frequently observed necessity when it comes to deconvolution estimators with super smooth error densities, compare \cite{ComteLacour2013}.
\begin{thm}[Data-driven choice of $\bm k$]\label{thm:pco:sv}
	Let $f_{\bm V}\in \IL^2(\IRp^2, \bm x^{\ushort{\bm 1}})$ and $(\bm{\mathrm{A}_0})- (\bm{\mathrm{A}_3})$ hold true.
	Then there exists $\chi_0\in \IRp$ such that for all $\chi \geq \chi_0$, 
	\begin{align*}\E_{f_{\bm Y}}^n(\|\widehat f_{\Delta,\widehat{\bm k}}-f_{\bm V}\|_{\bm x^{\ushort{\bm 1}}}^2) \leq &C\inf_{\bm k\in \mathcal K_n}\left(\|f_{\bm V}-f_{\bm V, \bm k}\|_{\bm x^{\ushort{\bm 1}}}^2 +\mathrm{pen}(\bm k)\right) + C(\mathfrak c)\Delta\log^6(n)+ \frac{C(g)}{n} \\
	&+\frac{C(\beta_{\bm V}) \log^2(n)}{n\Delta}\end{align*}
	where $C(\mathfrak c)$, $C(g)$ and $C(\beta_{\bm V})$  are positive constants only depending on ($\bm{\mathrm A_3}$), $g$ and $\beta_{\bm V}$. Then balancing $\Delta$ with repect to $n\in \IN$ leads to $\Delta:=\Delta_n=(n^{1/2}\log^2(n))^{-1}$ implying
	\begin{align*}
	\E_{f_{\bm Y}}^n(\|\widehat f_{\Delta_n,\widehat{\bm k}}-f_{\bm V}\|_{\bm x^{\ushort{\bm 1}}}^2) \leq C\inf_{\bm k\in \mathcal K_n}\left(\|f_{\bm V}-f_{\bm V, \bm k}\|_{\bm x^{\ushort{\bm 1}}}^2\right) + C(\mathfrak c, g, \beta_{\bm V}) \frac{\log^4(n)}{n}.
	\end{align*}
\end{thm}
We end this section by giving a short discussion for which values of $\bm c$ the stated Theorem \ref{pro:volatility} and \ref{thm:pco:sv} can be generalised.
\begin{rem}
	For values of $\bm c \neq \bm 1$ additional assumptions on the moments of $\bm X$ and $\bm U$ are needed, compare Proposition \ref{pr:consis} and Theorem \ref{thm:pco:2}. Since in the stochastic volatility model the distribution of $\bm U_1$ is known to follow a $\chi^2$-distribution in each direction, we can deduce restrictions on $\bm c$ to ensure that the estimator is well-defined. Then from $\E_g(\bm U_1^{\ushort{2({\bm c}-\bm 1})})< \infty$, we deduce $\bm c> \bm{3/4}$ in order to generalise the result of Theorem \ref{pro:volatility}. This excludes the case of an unweighted $\IL^2$ risk which corresponds to the case of $\bm c=(1/2, 1/2)^T$. To generalise Theorem \ref{thm:pco:sv}, we need that $\E(\bm U^{4({\bm c}-\bm 1)})< \infty$, leading to $\bm c > \bm{7/8}$.
\end{rem}
\section{Examples for volatility processes and numerical studies}
The following section is separated into two parts. In the first a collection of examples of volatility processes, fulfilling ($\bm {\mathrm A_{0}}$)-($\bm{\mathrm A_3}$), is given with a study of the upcoming bias terms and the rates of the fully data-driven anisotropic estimators $\widehat f_{\Delta, \widehat{\bm k}}$. In the second part, we will briefly illustrate the expected behaviour of the proposed estimator via a Monte-Carlo simulation study.
\subsection{Collection of volatility processes}
\paragraph*{Exponential of a bivariate Ornstein-Uhlenbeck process} Let $(\bm Z_t)_{t\geq 0}$ be the stationary solution of the stochastic differential equation
\begin{align}\label{eq:OUP}
d\bm Z_t = \begin{pmatrix} -9 & 1 \\ 0 & -7 \end{pmatrix} \bm Z_tdt + \begin{pmatrix} 3 & 1 \\ 0 & 2 \end{pmatrix} d\widetilde{\bm W}_t,
\end{align}
where $(\widetilde{\bm W}_t)_{t\geq 0}$ is a standard Brownian motion.
In this situation, the invariant density $f_{\bm Z}$ of the process $(\bm Z_t)_{t\geq 0}$ is given by	$f_{\bm Z}\sim \mathrm{N}_{(\bm 0, \bm \Sigma)}$ where $$\bm \Sigma= \frac{1}{7} \begin{pmatrix} 4 & 1 \\ 1& 2 \end{pmatrix}$$ 
and the process $(\bm Z_t)_{t\geq 0}$ is $\beta$-mixing with exponential decay, compare \cite{Schmisser2013a} Section 5.2 Examples.
Then, $\bm V_t:= (\exp(Z_{t,1}), \exp(Z_{t,2}))^T$ is stationary, $\beta$-mixing with density $f_{\bm V}$ given by
$$f_{\bm V}(\bm x)= \frac{\sqrt{7}}{2\pi x_1 x_2} \exp\left(-\log^2(x_1)+\log(x_1)\log(x_2)-2\log^2(x_2)\right), \quad \bm x \in \IRp^2.$$
 In other words, $\bm V_0$ follows a bivariate log normal distribution. Further, exploiting Proposition \ref{prop:ass:3}, we have that ($\bm{\mathrm A_0}$)-($\bm{\mathrm{A}_3}$) are fulfilled. By \cite{Brenner-Miguel2021}, the Mellin transform of $f_{\bm V}$ is then given by
$$\mathcal M_{\bm 1}[f_{\bm V}](\bm t)= \exp\left(-\frac{1}{2} \bm t^T \bm \Sigma \bm t\right), \quad \bm t \in \IR^2.$$
From this we can deduce that $\|f_{\bm V}-f_{\bm V , \bm k}\|_{\bm x^{\ushort{\bm 1}}}^2 \leq L(e^{-\frac{2}{7}k_1^{2}}+e^{-\frac{1}{7} k_2^2})$ for some numerical constant $L>0$. Then, a direct calculus and Theorem \ref{thm:pco:sv} implies for this particular case of $(\bm V_t)_{t\geq 0}$ 
$$\E_{f_{\bm Y}}^n(\|\widehat f_{\Delta_n,\widehat{\bm k}}-f_{\bm V}\|_{\bm x^{\ushort{\bm 1}}}^2) \leq C(\bm V, g) \frac{\log^4(n)}{n}.$$
\paragraph*{Bivariate Cox-Ingersoll-Ross process}
Let $\rho_1, \rho_2 \in \IN$ and set $\bm \rho = (\rho_1, \rho_2)^T$. Then we define the bivariate Cox-Ingersoll-Ross process with independent coordinates as the strictly stationary solution of 
\begin{align}\label{eq:CIR}
d\bm V_t = (2\bm \rho-\bm V_t )dt + \sqrt{2} \begin{pmatrix} \sqrt{V_{t,1}} & 0 \\ 0 & \sqrt{V_{t,2}} \end{pmatrix} d\widetilde{\bm W}_t,
\end{align}
in other words, each variate is a Cox-Ingersoll-Ross process. From the univariate case, we deduce that the process fulfills ($\bm {\mathrm A}_1$). To see ($\bm{\mathrm A_2}$), one exploits that we can construct an univariate, and thus also a bivariate, CIR process using the sums of the squared coordinates of a multivariate Ornstein-Uhlenbeck processes without drift, which is $\beta$-mixing with exponential decay. Thus ($\bm{\mathrm A_2}$) holds true. 
In this situation, the invariant density $f_{\bm V}$ is given by a Gamma distribution
$$ f_{\bm V}(x)= \frac{\bm x^{\ushort{\bm \rho-\bm 1}}}{\Gamma(\rho_1)\Gamma(\rho_2)} \exp(-x_1-x_2) \mathds 1_{\IRp^2}(\bm x), \quad \bm x\in \IRp^2.$$
Exploiting \cite{Gloter2000} Proposition 5.5, for $i\in \nset{2}$ we have $\E(\sup_{t\in [0,\Delta]} V_{t,i}^{-2}) \leq 2\E(V_{0,i}^{-2})<\infty$ for $\rho_1,\rho_2 \geq 3$, which together with Proposition \ref{prop:ass:3}, implies ($\bm{\mathrm A_3}).$ The Mellin transform of $f_{\bm V}$ is given by 
$$ \sM_{\bm 1}[f_{\bm V}](\bm t)= \frac{\Gamma(\rho_1+it_1)\Gamma(\rho_2+it_2)}{\Gamma(\rho_1)\Gamma(\rho_2)}, \quad \bm t\in \IR^2.$$
Thus applying the Stirling inequality for Gamma functions, compare \cite{AndrewsAskeyRoy1999} Corollary 1.4.4., $\|f_{\bm V}-f_{\bm V , \bm k}\|_{\bm x^{\ushort{\bm 1}}}^2 \leq L(\rho_1, \rho_2)(k_1^{2(\rho_1+1)}e^{-k_1\pi}+k_2^{2(\rho_2+1)}e^{-k_2\pi})$ for a constant $L(\rho_1,\rho_2)$ dependent on $\rho_1, \rho_2$. Then, we derive 
$$\E_{f_{\bm Y}}^n(\|\widehat f_{\Delta_n,\widehat{\bm k}}-f_{\bm V}\|_{\bm x^{\ushort{\bm 1}}}^2) \leq C(\bm V, g) \frac{\log^{2(\rho_1\vee\rho_2+1)}(n)}{n}.$$
\paragraph*{Exponential of a bivariate Cox-Ingersoll-Ross process}
We consider $\bm V_t:=(\exp(Z_{t,1}), \exp(Z_{t,2}))^T$ where $(\bm Z_t)_{t\geq 0}$ is an bivariate Cox-Ingersoll Ross process with $\bm \rho\in \IN^2$. Instantly, the properties ($\bm{\mathrm A_1}$) and ($\bm{\mathrm A_2})$ are given. The invariant density $f_{\bm V}$ is here given by the density of a Log-Gamma distribution that is,
$$f_{\bm V}(\bm x)=\frac{\log^{\rho_1-1}(x_1)\log^{\rho_2-1}(x_2)}{\Gamma(\rho_1)\Gamma(\rho_2)} \bm x^{\ushort{-\bm 2}} \mathds 1_{(1,\infty)^2}(\bm x), \quad \bm x \in \IRp^2.$$
For $(\bm{\mathrm A_3})$ we again use Proposition \ref{prop:ass:3}. The corresponding Mellin transform is then given by
$$ \sM_{\bm 1}[f_{\bm V}](\bm t)= (\bm 1-i\bm t)^{\ushort{-\bm \rho}}, \quad \bm t\in \IR^2,$$
with $\|f_{\bm V}-f_{\bm V , \bm k}\|_{\bm x^{\ushort{\bm 1}}}^2 \leq L(\rho_1, \rho_2)(k_1^{-2\rho_1+1}+k_2^{-2\rho_2+1})$ for $\bm k\in \IRp^2$. Further, we have
$$\E_{f_{\bm Y}}^n(\|\widehat f_{\Delta_n,\widehat{\bm k}}-f_{\bm V}\|_{\bm x^{\ushort{\bm 1}}}^2) \leq C(\bm V, g) \log(n)^{-2(\rho_1 \wedge \rho_2)+1}.$$
\paragraph*{Comment}
Based on Theorem \ref{thm:pco:sv}, it was clear that the rate of the fully data-driven estimator cannot achieve a rate better than $\log^4(n)/\sqrt{n}$. In the case of the exponential of an Ornstein-Uhlenbeck process, we have seen that the fast decay of $\sM_{\bm 1}[f_{\bm V}]$ implies that the estimator achieves this rate, while for the CIR processes, a slight disgression of the rate is observed. \\
In the case of exponential of a CIR process, the rate is of logarithmic decay which is typical for super smooth errors and densities with polynomial decaying Mellin transform. \\
\label{sec_numerical}
\subsection{Numerical simulation}
We illustrate the performance of the estimator $\widehat f_{\Delta, \widehat{\bm k}}$, defined in \eqref{def:est} and \eqref{eq:pco:sv}, using a Monte-Carlo simulation. To do so, we sample for $\Delta= 0.01$ fixed and varying sample sizes $n\in \IN$ from an exponential Ornstein-Uhlenbeck process, \eqref{eq:OUP}, $(\bm V_t)_{t\in [0,\Delta n]}$ and calculate the scaled integrated volatilities $(\overline{\bm V}_{j})_{j\in \nset{n}}$. Here, the sampling from the process and the calculation of the upcoming integral are solved by numerical discretisation. \\
In Figure \ref{figure:1}, we compare the estimator $\widehat f_{\Delta, \widehat{\bm k}}$ in the volatility model with the estimator $\widehat f_{\widehat{\bm k}}$ of \cite{Brenner-Miguel2021} based on the direct observation $(\overline {\bm V}_j)_{j\in \nset{n}}$, that is without noise. 
\begin{minipage}{\textwidth}
\begin{minipage}{0.45\textwidth}
	\begin{minipage}{\textwidth}
			\includegraphics[width=0.44\textwidth,height=30mm]{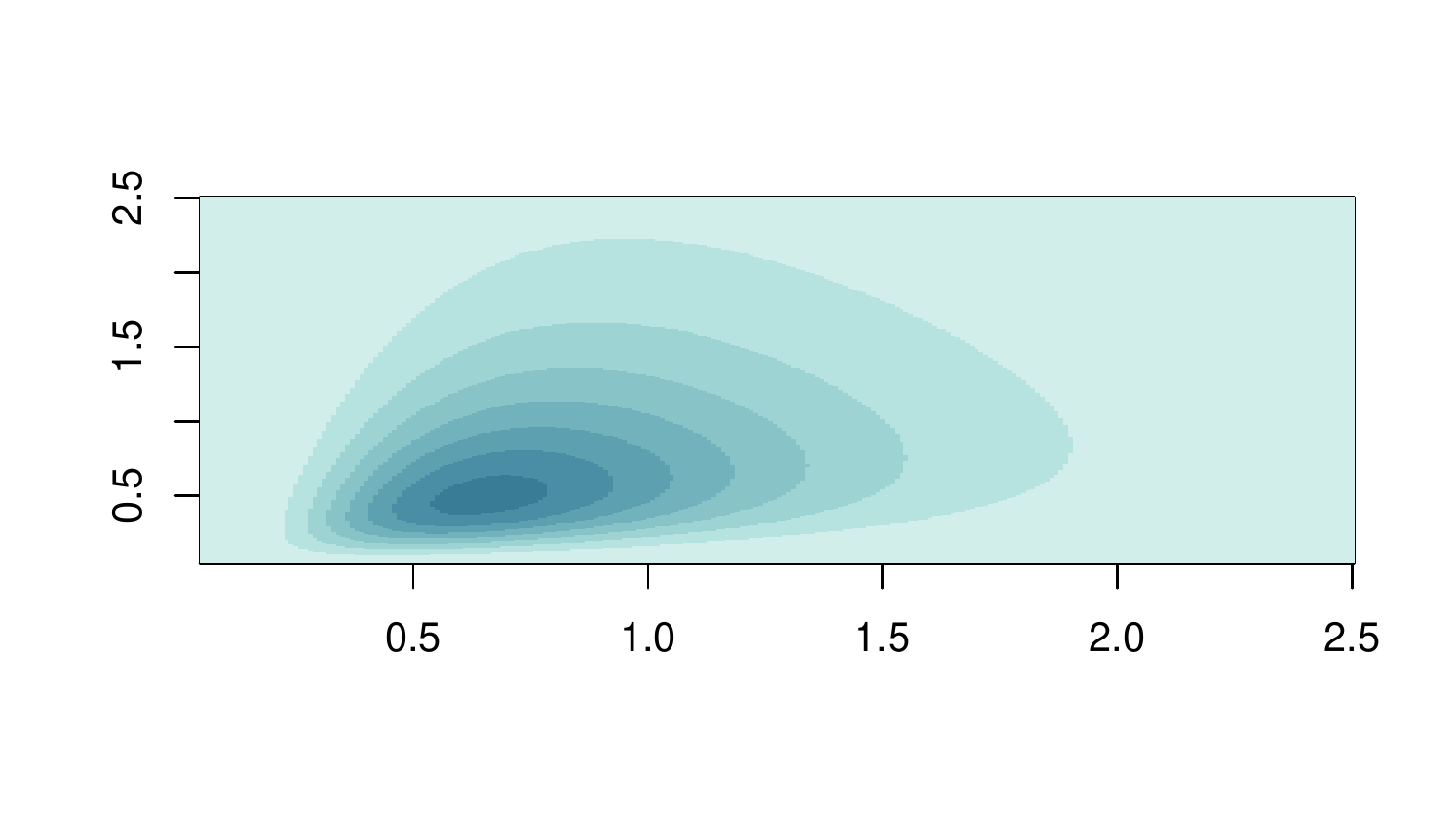}
			\includegraphics[width=0.44\textwidth,height=30mm]{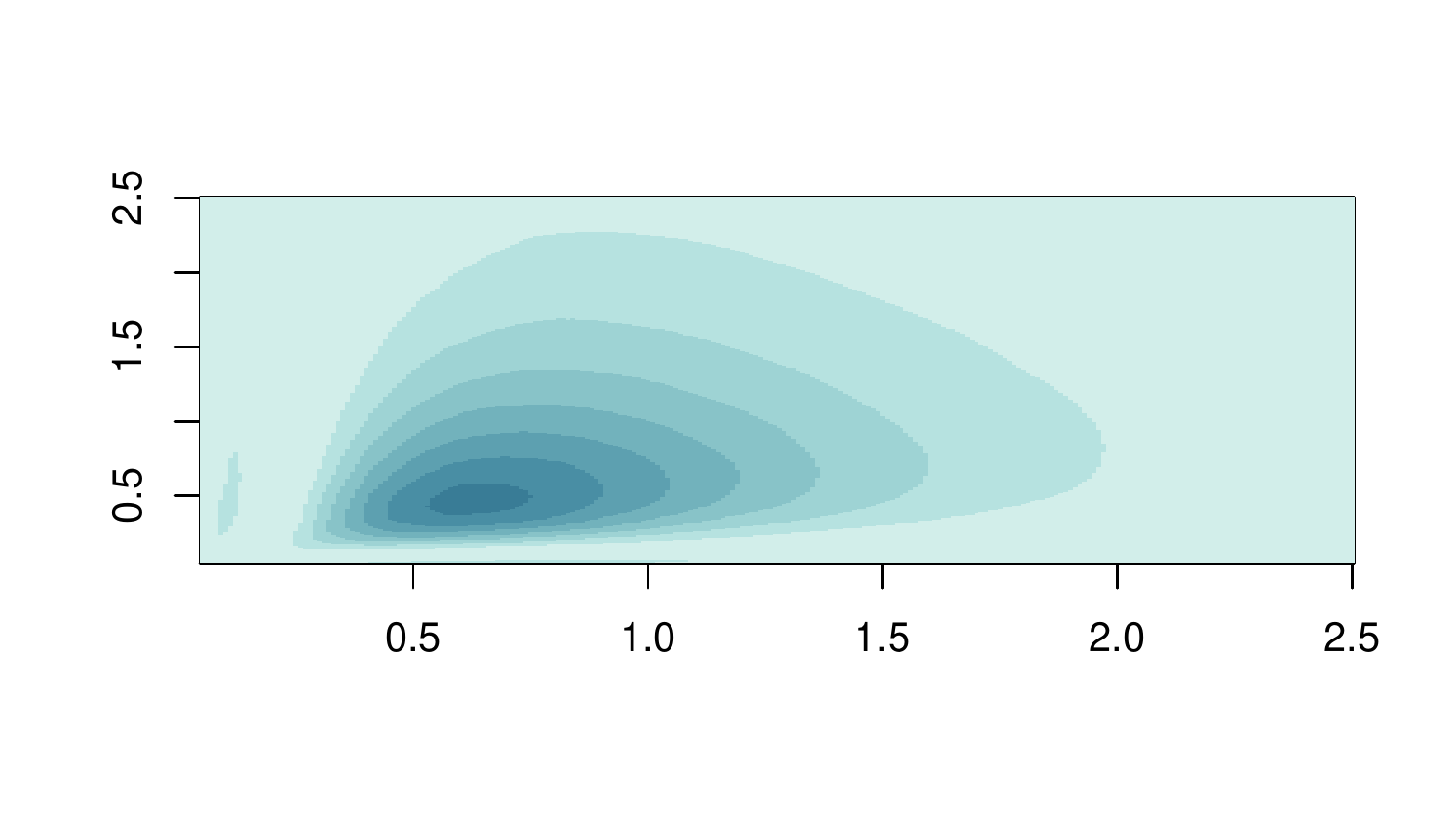}
		\end{minipage}
		\begin{minipage}{0.45\textwidth}
				\includegraphics[width=\textwidth,height=30mm]{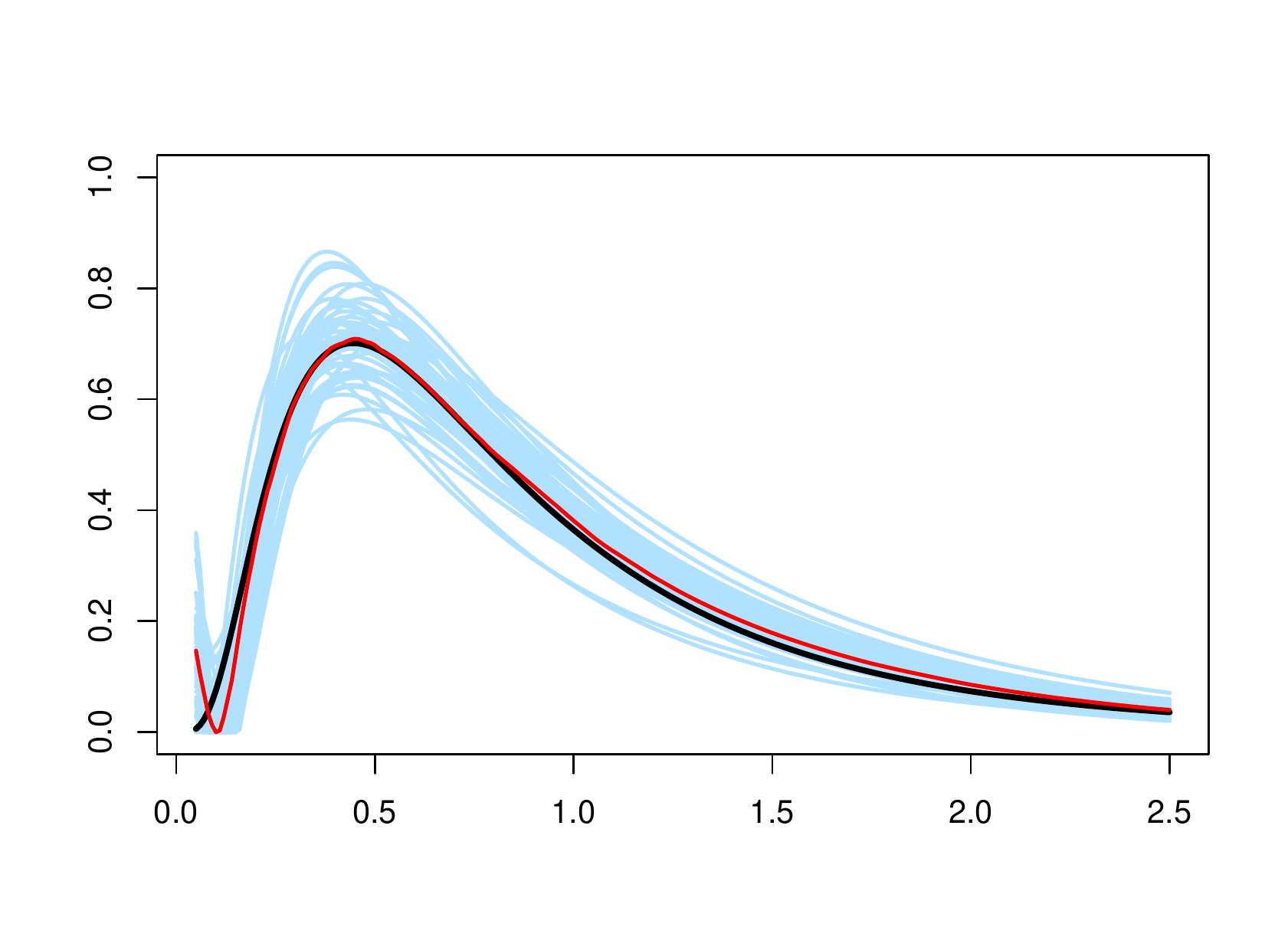}
			\end{minipage}
			\begin{minipage}{0.45\textwidth}
				\includegraphics[width=\textwidth,height=30mm]{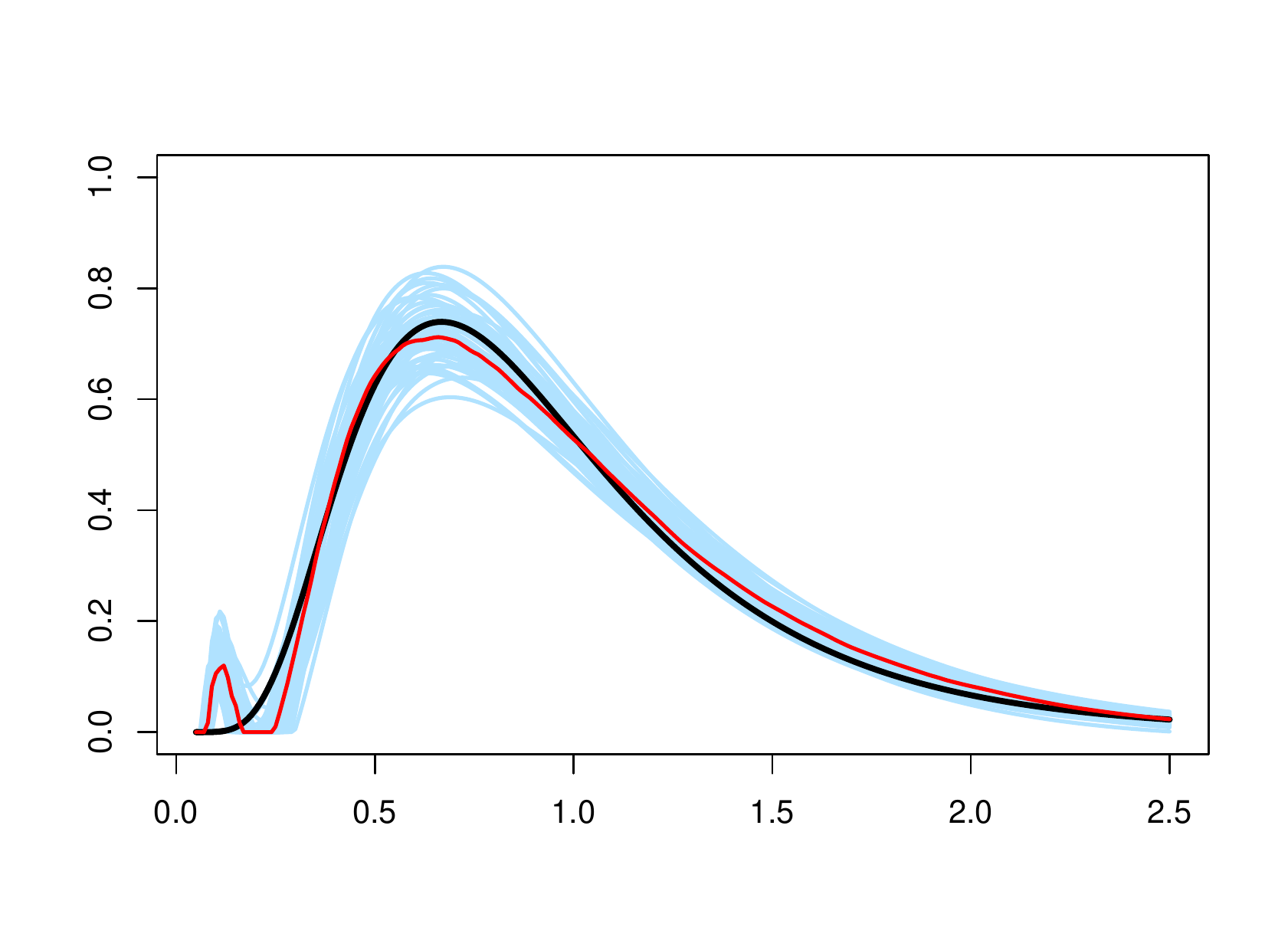}
		\end{minipage}
\end{minipage}\hspace*{0.5cm}
\begin{minipage}{0.45\textwidth}
		\begin{minipage}{\textwidth}
		\includegraphics[width=0.44\textwidth,height=30mm]{ln_vol_true.pdf}
		\includegraphics[width=0.44\textwidth,height=30mm]{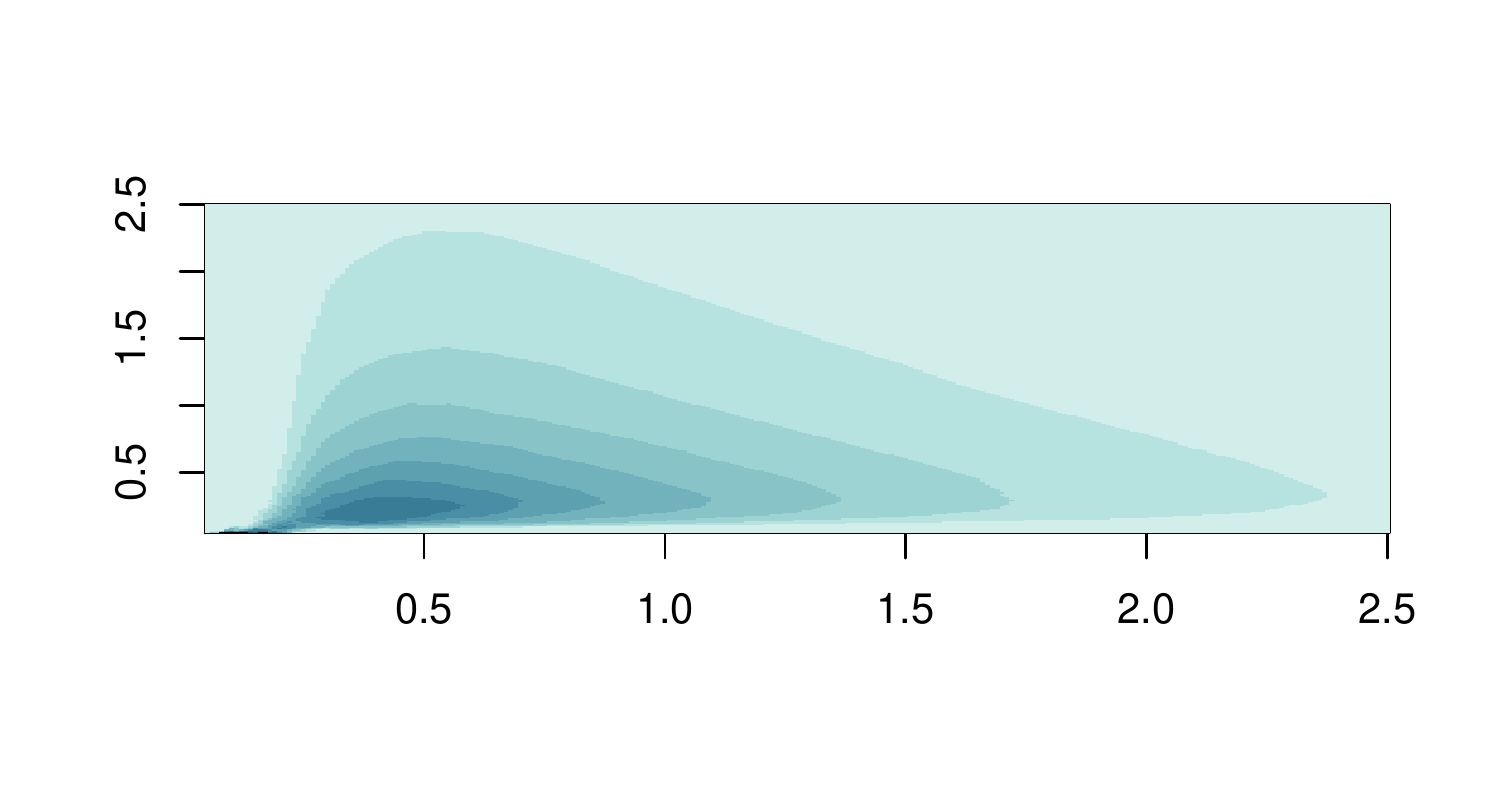}
	\end{minipage}
	\begin{minipage}{0.45\textwidth}
		\includegraphics[width=\textwidth,height=30mm]{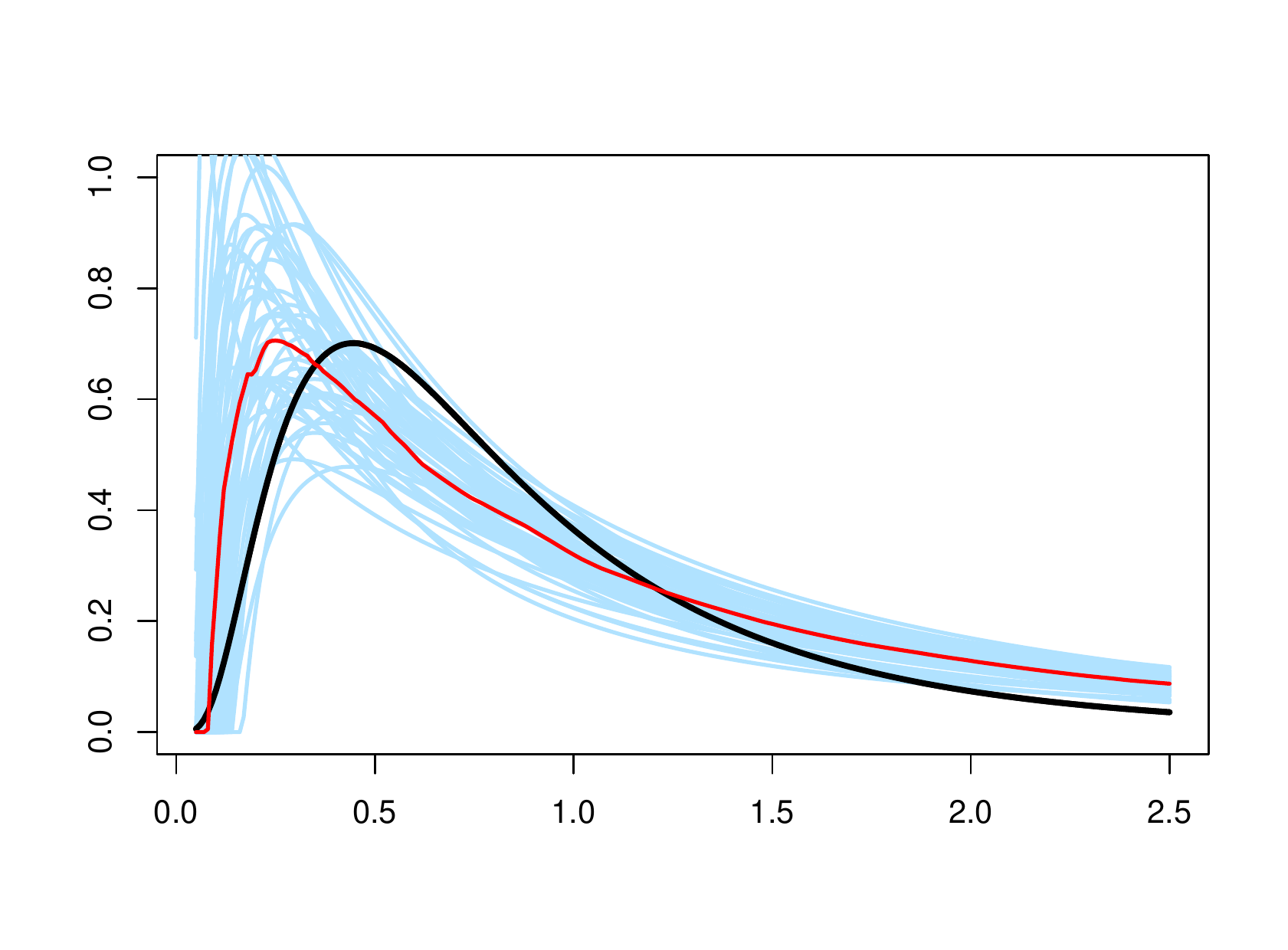}
	\end{minipage}
	\begin{minipage}{0.45\textwidth}
		\includegraphics[width=\textwidth,height=30mm]{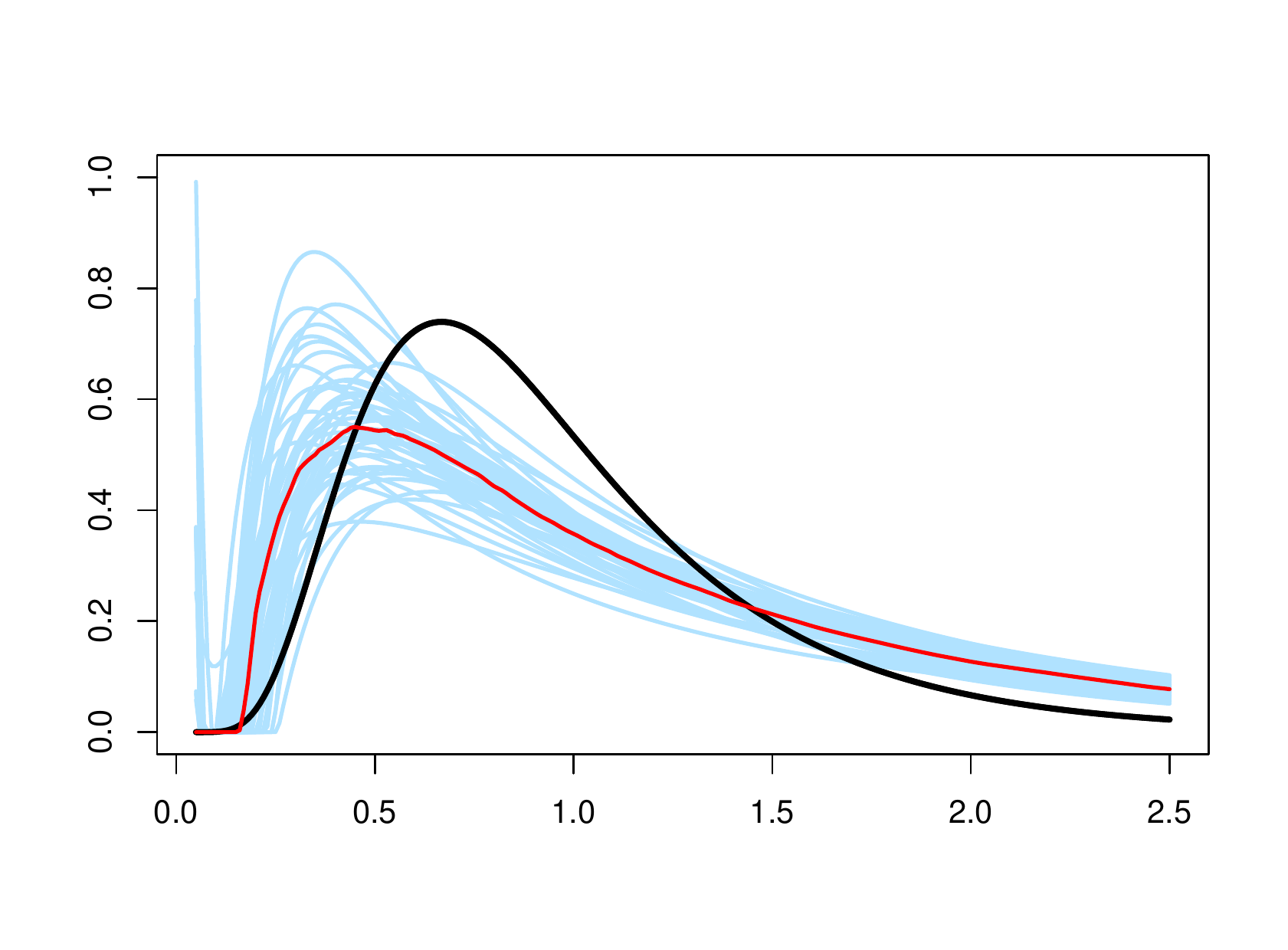}
	\end{minipage}
\end{minipage}
	\captionof{figure}{\label{figure:1} Estimators $\widehat f_{\widehat {\bm k}}$ (left) and $\widehat f_{\widehat{\bm k}, \Delta}$ (right) depicted for 
	50  Monte-Carlo simulations with  $n=5000$ based on $(\overline{\bm V}_j)_{j\in \nset{n}}$, respectively based on $(\bm Y_j)_{j\in \nset{n}}$. Top plots: the true density (left) and the pointwise median of the estimators (right). Bottom plots:  sections for $x=0.54$ (right) and $y=0.54$ (left) with  true density $f$ (black curve) and pointwise empirical median (red curve) of the 50 estimates.}
\end{minipage}\\
In Figure \ref{figure:1}, one sees the impact of the noise on the performance of the estimator. Focusing on the pointwise median, the remaining bias is clearly observable which is consistent with the theory. Figure \ref{figure:2} illustrates the improvement of the behavior of the estimator for increasing sample size.
\begin{minipage}{\textwidth}
	\begin{minipage}{0.45\textwidth}
		\begin{minipage}{\textwidth}
			\includegraphics[width=0.44\textwidth,height=30mm]{ln_vol_true.pdf}
			\includegraphics[width=0.44\textwidth,height=30mm]{ln_vol_5000.pdf}
		\end{minipage}
		\begin{minipage}{0.45\textwidth}
			\includegraphics[width=\textwidth,height=30mm]{ln_5000_vol_x.pdf}
		\end{minipage}
		\begin{minipage}{0.45\textwidth}
			\includegraphics[width=\textwidth,height=30mm]{ln_5000_vol_y.pdf}
		\end{minipage}
	\end{minipage}\hspace*{0.5cm}
	\begin{minipage}{0.45\textwidth}
		\begin{minipage}{\textwidth}
			\includegraphics[width=0.44\textwidth,height=30mm]{ln_vol_true.pdf}
			\includegraphics[width=0.44\textwidth,height=30mm]{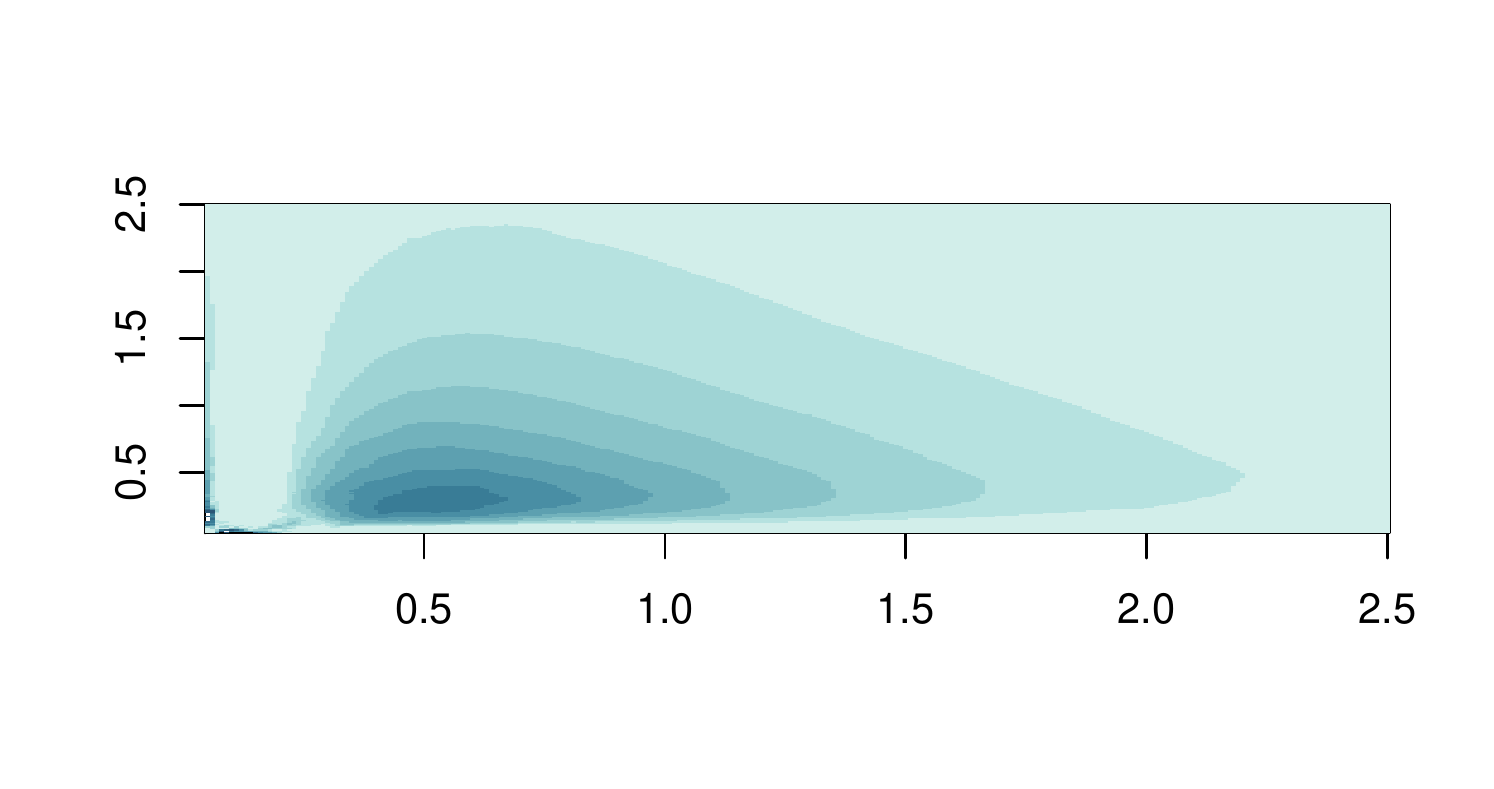}
		\end{minipage}
		\begin{minipage}{0.45\textwidth}
			\includegraphics[width=\textwidth,height=30mm]{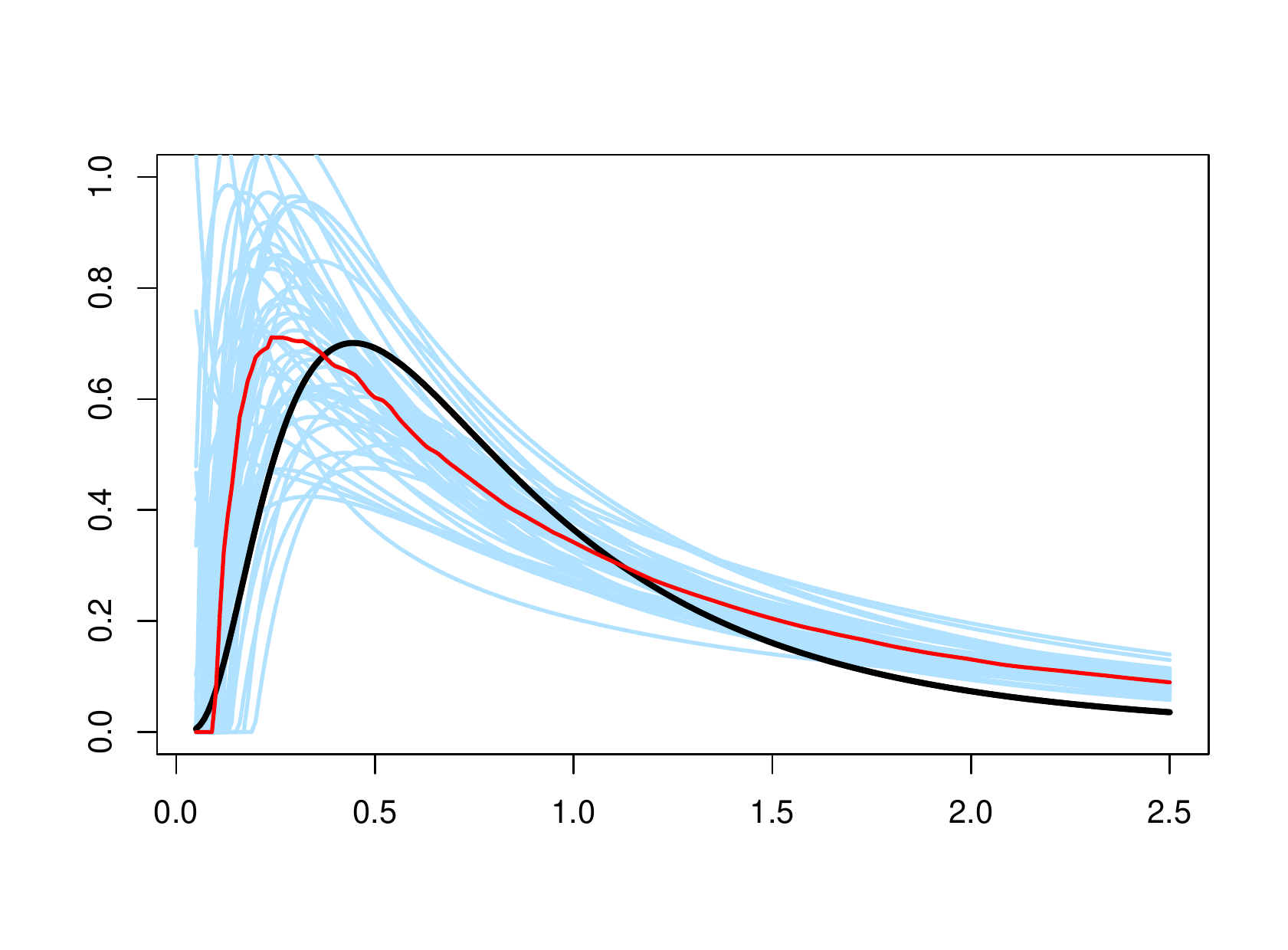}
		\end{minipage}
		\begin{minipage}{0.45\textwidth}
			\includegraphics[width=\textwidth,height=30mm]{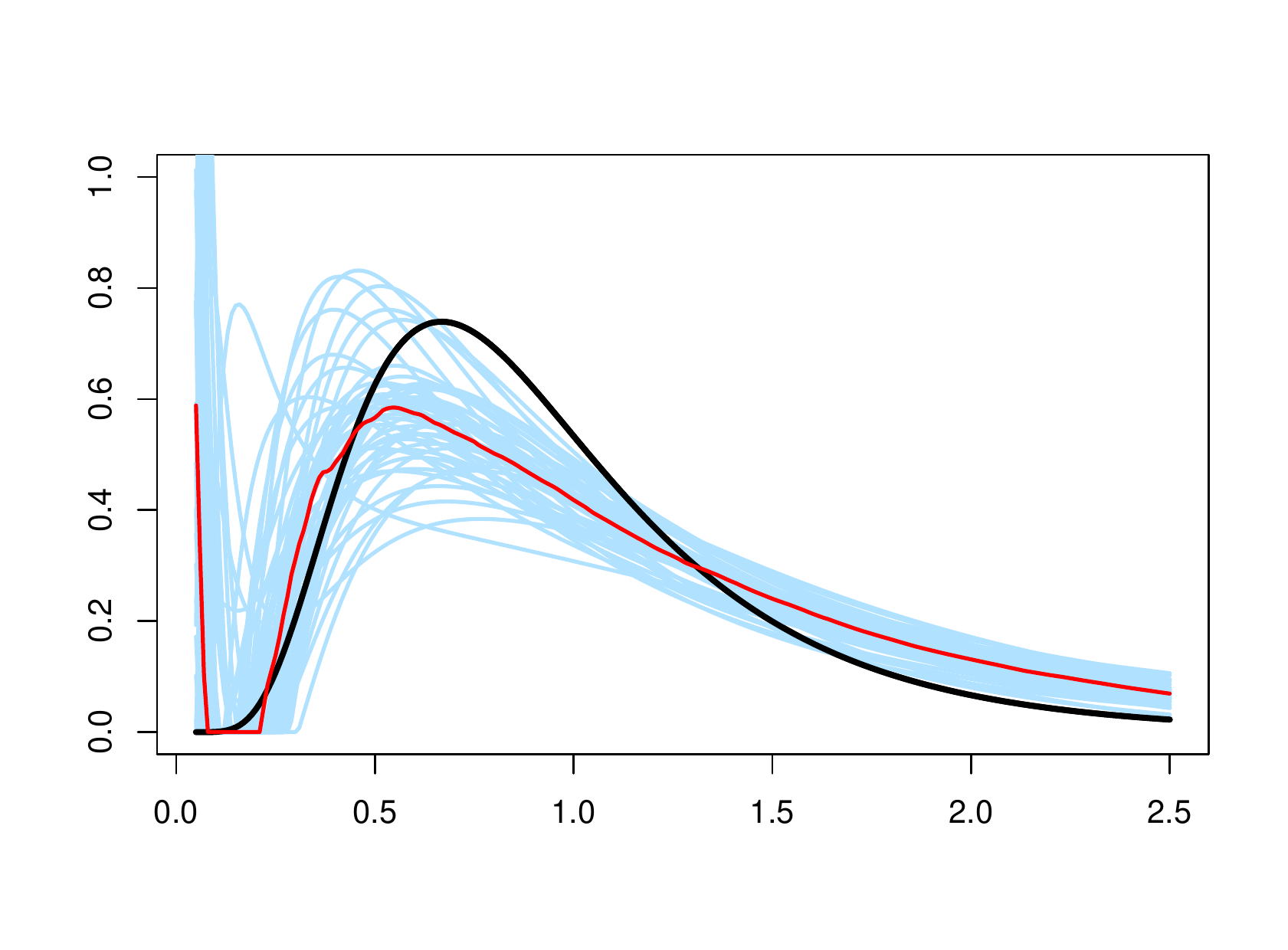}
		\end{minipage}
	\end{minipage}
	\captionof{figure}{\label{figure:2} Estimator $\widehat f_{\widehat{\bm k}, \Delta}$ depicted for 
		50  Monte-Carlo simulations with   $n=5000$ (left) and $n=20000$ (right) based on $(\bm Y_j)_{j\in \nset{n}}$. Top plots: true density (left) and the pointwise median of the estimators (right). Bottom plots: sections for $x=0.54$ (right) and $y=0.54$ (left) with true density $f$ (black curve) and  pointwise empirical median (red curve) of the 50 estimates.}
\end{minipage}\\
\paragraph*{Comment}
The simulation study implies the reasonable behavior of the estimator. For increasing sample size the error of the estimator is decaying. Furthermore, it seems that the underlying dependence has a negligible effect on the rate compared to the super smooth error densities. This observation is consistent with the theoretical results of Theorem \ref{pro:volatility} and \ref{thm:pco:sv}. 
\section{Multiplicative measurement error model for stationary processes}\label{sec_stat_pro}
In the following section we consider the estimation of the density $f: \IR_+^2 \rightarrow \IR_+$ of a positive, bivariate random vector $\bm X=(X_1, X_2)^T$ based on a strictly stationary sample of $\bm X$ under multiplicative measurement errors, that is, we consider the observations 
\begin{align*}
\bm Y_i:= \begin{pmatrix} Y_{i,1} \\ Y_{i,2} \end{pmatrix} =   \begin{pmatrix} X_{i,1} U_{i,1} \\ X_{i,2} U_{i,2}\end{pmatrix}  = \bm X_i \bm U_i,\quad  i\in \nset{n},
\end{align*}
where $(\bm X_i)_{i\in \llbracket n \rrbracket}$ are sampled from a strictly stationary process with stationary density given by $f$ and $(\bm U_i)_{i\in \llbracket n \rrbracket}$ is an i.i.d. sequence drawn from the error density $g:\IR_+^2 \rightarrow \IR_+$. 
To do so, we will borrow ideas from \cite{Brenner-Miguel2021} and \cite{Brenner-MiguelPhandoidaen2022}. In comparison to \cite{Brenner-Miguel2021} and \cite{Brenner-MiguelPhandoidaen2022}, where smooth error densities has been considered, our main focus will lie on super smooth error densities.  Before building our estimator, let us briefly summarise main properties of the Mellin transform presented in \cite{Brenner-Miguel2021}.
\paragraph*{The Mellin transform} Let $\bm c\in \IRp^2$. For two functions $h_1,h_2\in \mathbb L^1(\IRp^2, \bm x^{\ushort{\bm c-\bm 1}})$ we define the multiplicative convolution $h_1*h_2$ of $h_1$ and $h_2$ by
\begin{align}
(h_1*h_2)(\bm y)=\int_{\IRp^2} h_1(\bm y/\bm x) h_2(\bm x) \bm x^{\ushort{-\bm 1}} d\bm x, \quad \bm y\in \mathbb R^2.
\end{align}
It can be shown,  $h_1*h_2$ is well-defined, $h_1*h_2=h_2*h_1$ and $h_1*h_2 \in \mathbb L^1(\IRp^2,\bm x^{\ushort{\bm c-\bm1}})$. A proof sketch of this property and the following results can be found in \cite{Brenner-Miguel2021}. If additionally $h_1\in \mathbb L^2(\IRp^2, \bm x^{\ushort{2\bm c-\bm 1}})$ then $h_1*h_2 \in \mathbb L^2(\IRp^2,\bm x^{\ushort{2\bm c-\bm 1}})$. One key property of the Mellin transform, which makes it so appealing for the use of multiplicative deconvolution, is the so-called convolution theorem, that is, for $h_1, h_2\in \mathbb L^1(\IRp^2,\bm x^{\ushort{\bm c-\bm 1}})$, 
\begin{align}
\mathcal M_{\bm c}[h_1*h_2](\bm t)=\mathcal M_c[h_1](\bm t) \mathcal M_c[h_2](\bm t), \quad \bm t\in \mathbb R^2.
\end{align}
By construction, the operator $\mathcal M_{\bm c}: \IL^2(\IRp^2,\bm x^{\ushort{2\bm c-\bm 1}}) \rightarrow \IL^2(\IR^2)$ is an isomorphism. Denoting by $\mathcal M_{\bm c}^{-1}: \IL^2(\IR^2) \rightarrow \mathcal \IL^2(\IRp^2,\bm x^{\ushort{2\bm c-\bm 1}})$ its inverse, we can state that if additional to $H\in \IL^2(\IR^2)$, $H\in \IL^1(\IR^2)$ holds true, then $\mathcal M_{\bm c}^{-1}[H]$ can be expressed explicitly by
\begin{align}\label{eq:Mel:inv}
\sM_{\bm c}^{-1}[H](\bm x)= \frac{1}{(2\pi)^2 } \int_{\IR^2} \bm x^{\ushort{-\bm c-i\bm t}} H(\bm t) d\bm t, \quad \text{ for any } \bm x\in \IRp^2.
\end{align} 
Furthermore, we can directly show that a Plancherel-type equation, respectively Parseval-type equation, is valid for the Mellin transform. For all $ h_1, h_2 \in \IL^2(\IRp^d,\bm x^{\ushort{2\bm c-\bm 1}})$ holds 
\begin{align}\label{eq:Mel:plan}
\hspace*{-1cm}\langle h_1, h_2 \rangle_{\bm x^{\ushort{2\bm c-\bm 1}}} = \frac{1}{4\pi^2} \langle \sM_{\bm c}[h_1], \sM_{\bm c}[h_2] \rangle_{\IR^2} 
\end{align} 
and thus $\| h_1\|_{\bm x^{\ushort{2\bm c-\bm 1}}}^2=(4\pi)^{-1} \|\sM_{\bm c}[h_1]\|_{\IR^2}^2$.
\paragraph*{Estimation strategy} Let $\bm c\in \IR^2$ and  $\IE_{f_{\bm Y}}(\bm Y_1^{\ushort{\bm c-\bm 1}})<\infty$ and $f\in \IL^2(\IRp^2, \bm x^{\ushort{2\bm c-\bm 1}})$. Then, we define for the spectral cut-off estimator $\widehat f_{\bm k}$, studied in \cite{Brenner-Miguel2021}, 
\begin{align}\label{eq:est:def}
\widehat f_{\bm k}(\bm x)= \frac{1}{4\pi^2} \int_{[-\bm k, \bm k]} \bm x^{\ushort{-\bm c-i\bm t}} \frac{\widehat{\mathcal M}_{\bm c}(\bm t)}{\sM_{\bm c}[g](\bm t)} d\bm t,\quad \text{with }\widehat{\sM}_{\bm c}(\bm t) := \frac{1}{n}\sum_{j\in \nset{n}} \bm Y_j^{\ushort{\bm c-\bm 1+i\bm t}}, \quad 
\end{align}
for $\bm x\in \IRp^2.$ To ensure that the estimator is well-defined, we assume that $g$ fulfills
\begin{align*}
\forall \bm t \in \IR^2: \sM_{\bm c}[g](\bm t)\neq 0 \text{ and } \forall \bm k \in \IRp^2: \int_{[-\bm k, \bm k]} |\sM_{\bm c}[g](\bm t)|^{-2} d\bm t <\infty \tag{\textbf{[G0]}}.
\end{align*}
\begin{rem}
	Assumption \textbf{[G0]} is not unusual in context of deconvolution problems, compare \cite{Meister2009}. Examples of multivariate density, which fullfils the assumption \textbf{[G0]} are presented in \cite{Brenner-Miguel2021}.
\end{rem}
The following proposition is a generalisation of the results in \cite{Brenner-Miguel2021} for strictly stationary data. Its proof is postponed to Appendix \ref{a:stat_pro}.

	\begin{enumerate}
		\item Technische Einführung der Mellin transformierter für Multivariate
	\end{enumerate}

\begin{prop}[Upper bound of the risk]\label{pr:consis}
	Let $f\in \IL^2(\IRp^d, \bm x^{\ushort{2\bm c-\bm 1}})$, $\mu_{\bm Y}:= \E_{f_{\bm Y}}(\bm Y_1^{\ushort{2\bm c- \bm 2}})< \infty$ and $g$ fulfill \textbf{[G0]}. Then, for any $\bm k\in \IRp^2$,
	\begin{align} \label{eq:consis}
	\hspace*{-0.25cm}\E_{f_Y}^n(\|f-\widehat f_{\bm k}\|_{\bm x^{\ushort{2\bm c-\bm 1}}}^2) \leq \|f-f_{\bm k}\|_{\bm x^{\ushort{2\bm c-\bm 1}}}^2+ \frac{\mu_{\bm Y} \Lambda_g(\bm k)}{n}+ \frac{1}{4\pi^2} \int_{[-\bm k, \bm k]} \Var_f^n(\widehat \sM_{\bm X}(\bm t))d\bm t
	\end{align}
	where $\Lambda_g(\bm k):= (4\pi^2)^{-1}\int_{[-\bm k, \bm k]} |\sM_{\bm c}[g](\bm t)|^{-2} d\bm t$ and $\widehat \sM_{\bm X}(\bm t):= n^{-1} \sum_{j\in \nset{n}} \bm X_{j}^{\ushort{\bm c-\bm 1+i\bm t}}$. 
\end{prop}
Let us comment on the bound in Proposition \ref{pr:consis}. The upper bound of the risk, for $\bm k\in \IRp^2$, consists of the usual squared bias term $\|f-f_{\bm k}\|_{\bm x^{\ushort{2\bm c-\bm 1}}}^2$ and a decomposition of the variance term $\E_{f_{\bm Y}}^n(\|f_{\bm k}-f_{\bm k}\|_{\bm x^{\ushort{2\bm c-\bm 1}}}^2)$. We decomposed it  into an inverse problem term $\mu_{\bm Y}\Delta_g(\bm k)n^{-1}$, which also appears in \cite{Brenner-Miguel2021} and an dependency term $(4\pi^2)\int_{[-\bm k, \bm k]} \Var_f^n(\widehat \sM_{\bm X}(\bm t))d\bm t$, which is consistent with the result of \cite{Brenner-MiguelPhandoidaen2022}. \\
Nevertheless, it is clear to seen that the squared bias term is decreasing for $\bm k\in \IRp^2$ componentwise increasing, while the variance term is increasing. A choice of $\bm k\in \IRp^2$ is without further information of the unknown density non-trivial. Therefore we propose in the next paragraph a fully data-driven choice of $\bm k\in \IRp^2$, that is a choice, which is only dependent on the sample $(\bm Y_j)_{j\in \nset{n}}$ without further knowledge about the density $f$.
\paragraph*{Data-driven choice of $\bm k\in \IRp^2$}
We restrict ourselves to the case $\IP^{\bm U_1} = \Gamma_{(1/2,1/2)}^{\otimes 2}$, motivated by the stochastic volatility model, for a simple display of results.\\
First we reduce the space of possible cut-off parameters to $\mathcal K_n:=\{\bm k \in \nset{\log(n)} \times \nset{\log(n)} : \Lambda_g(\bm k) \leq n\}$. Then we define the model selection method for $\chi>0$ and $\widehat{\mu}_{\bm Y}:=n^{-1}\sum_{j\in \nset{n}} \bm Y_j^{\ushort{2(\bm c-\bm 1)}}$ by
\begin{align}\label{eq:pco:2}
\widehat{\bm k}:=\argmin_{\bm k\in \mathcal K_n} -\|\widehat f_{\bm k}\|_{\bm x^{\ushort{2\bm c-\bm 1}}}^2 +\widehat{\mathrm{pen}}(\bm k), \quad \widehat{\mathrm{pen}}(\bm k):= \chi\widehat{\mu}_{\bm Y}k_1k_2\frac{\Lambda_{g}(\bm k)}{n}. 
\end{align}
Compared to the penalty in \cite{Brenner-MiguelComteJohannes2021}, we see that the term $\mathrm{pen}(\bm k):= \E(\widehat{\mathrm{pen}}(\bm k)) = \chi \mu_{\bm Y} k_1k_2 \Lambda_g(\bm k) n^{-1}$ is not of the same order of the variance term. This overestimation of the variance term for supersmooth error densities is commonly found in the deconvolution literature.
\begin{thm}[Data-driven choice of $\bm k$]\label{thm:pco:2}
	Let $f\in \IL^2(\IRp^2, \bm x^{\ushort{2\bm c-\bm 1}})$ and $\E(\bm Y^{\ushort{4_1(\bm c-\bm 1)}})<\infty.$ 
	Then there exists $\chi_0\in \IRp$ such that for all $\chi \geq \chi_0$, 
	\begin{align*}\E(\|\widehat f_{\widehat{\bm k}}-f_{\bm K_n}\|_{\bm x^{\ushort{2\bm c-\bm 1}}}^2) \leq 3&\inf_{\bm k\in \mathcal K_n} \left(\|f_{\bm K_n}-f_{\bm k}\|_{\bm x^{\ushort{2\bm c-\bm 1}}}^2+ \mathrm{pen}(\bm k)\right) + C(g) \frac{\mu_{\bm X}}{n}\\
	&+C(g, \chi ) \frac{\E(\bm X_1^{\ushort{4(\bm c-\bm 1)}})\log(n)^2}{\mu_{\bm X} n}
	+\int_{[-\bm K_n, \bm K_n]} \Var_f^n(\widehat{\sM}_{\bm X}(\bm t))d\bm t\\
	&+C(g, \chi)\big( \frac{\Var(\widehat\mu_{\bm X})\log(n)^2}{\mu_{\bm X}}\big) \end{align*}
	where $C(g)$, respectively $C(g,\chi)$ are positive constant only depending on $g$, respectively $g$ and $\chi$ and $\bm K_n:=(\lfloor \log(n)\rfloor, \lfloor \log(n)\rfloor)^T$. Furthermore, we directly deduce that
	\begin{align*}\E(\|\widehat f_{\widehat{\bm k}}-f\|_{\bm x^{\ushort{2\bm c-\bm 1}}}^2) \leq 11&\inf_{\bm k\in \mathcal K_n} \left(\|f-f_{\bm k}\|_{\bm x^{\ushort{2\bm c-\bm 1}}}^2+ \mathrm{pen}(\bm k)\right) + C(g) \frac{\mu_{\bm X}}{n}\\
	&+C(g, \chi ) \frac{\E(\bm X_1^{\ushort{4(\bm c-\bm 1)}})\log(n)^2}{\mu_{\bm X} n}
	+\int_{[-\bm K_n, \bm K_n]} \Var_f^n(\widehat{\sM}_{\bm X}(\bm t))d\bm t\\
	&+C(g, \chi)\big( \frac{\Var(\widehat\mu_{\bm X})\log(n)^2}{\mu_{\bm X}}\big). \end{align*}
\end{thm}
\begin{rem}
	It is worth stressing out that the first inequality of Theorem \ref{thm:pco:2} still holds true even without
	the assumption $f\in \IL^2(\IR_+^2,\bm x^{\ushort{2\bm c-\bm 1}})$. Indeed, by the abuse of notation $\sM_{\bm c}[f](\bm t):=\E(\bm X_1^{\ushort{\bm c-\bm 1+i\bm t}}), \bm t\in \IR_+^2,$ we can ensure that the functions $f_{\bm k}, \bm k\in \mathcal K_n$, are still well-defined with $\E_{f_Y}^n(\widehat f_{\bm k})=f_{\bm k}$. In other words, it is sufficient to use the $\IL^1$ notion of the Mellin transform since we only consider the distance between our estimator $\widehat f_{\widehat{\bm k}}$ and $f_{\bm K_n}$ and are therefore not in need of the assumption $f\in \IL^2(\IR_+, \bm x^{\ushort{2\bm c-\bm 1}})$. 
\end{rem}

    \paragraph*{Acknowledgement}
    %The authors would like to thank the anonymous referees, an Associate
    %Editor and the Editor for their constructive comments that improved the
    %quality of this paper. 
    This work is supported by the Deutsche Forschungsgemeinschaft (DFG, German Research Foundation) under Germany’s Excellence Strategy EXC 2181/1 - 390900948 (the Heidelberg STRUCTURES Excellence Cluster) and by the Research Training Group ”Statistical Modeling of Complex Systems”.
\newpage
\section{Appendix}
\label{sec_append}

\subsection{Useful inequalities}

The following inequality is due to
\cite{Talagrand1996}, the formulation of the first part  can be found for
example in \cite{KleinRio2005}.

\begin{lem}(Talagrand's inequality)\label{tal:re} Let
	$Z_1,\dotsc,Z_n$ be independent $\mathcal Z$-valued random variables and let $\bar{\nu}_{h}=n^{-1}\sum_{i=1}^n\left[\nu_{h}(Z_i)-\IE\left(\nu_{h}(Z_i)\right) \right]$ for $\nu_{h}$ belonging to a countable class $\{\nu_{h},h\in\sH\}$ of measurable functions. Then, for all $\varepsilon>0$
	\begin{align}
	&\IE\left(\sup_{h\in\sH}|\bar{\nu}_{h}|^2-2(1+2\varepsilon)\Psi^2\right)_+\leq C \left[\frac{\tau}{n}\exp\left(\frac{-K_1\varepsilon n\Psi^2}{\tau}\right)+\frac{\psi^2}{C_{\varepsilon}^2n^2}\exp\left(\frac{-C_{\varepsilon} K_2 \sqrt{\varepsilon} n \Psi}{\psi}\right) \right]\label{tal:re1} 
	\end{align}
	with numerical constants $C_{\varepsilon}:=\sqrt{1+\varepsilon}-1$ and $C>0$ and where
	\begin{equation*}
	\sup_{h\in\sH}\sup_{z\in\mathcal Z}|\nu_{h}(z)|\leq \psi,\qquad \IE(\sup_{h\in\sH}|\overline{\nu}_{h}|)\leq \Psi,\qquad \sup_{h\in \sH}\frac{1}{n}\sum_{j=1}^n \Var(\nu_{h}(Z_i))\leq \tau.
	\end{equation*}
\end{lem}
The key statement regarding $\beta$-mixing processes is delivered by the proposed variance bound derived by \cite{asin} after Lemma 4.1 of the same work. Their approach is based on the original idea of \cite[Theorem 2.1]{viennet}.
\begin{lem} \label{lem:beta_mix_var_bound}
	Let $(Z_j)_{j \in \IZ}$ be a strictly stationary process of real-valued random variables with common marginal distribution $\IP$. There exists a sequence $(b_k)_{k \in \IN}$ of measurable functions $b_k:\IR \to [0,1]$ with $\IE_{\IP}[b_k(Z_0)] = \beta(Z_0,Z_k)$ such that for any measurable function $h$ with $\IE[|h(Z_0)|^2] < \infty$ and $b = \sum_{k=1}^\infty (k+1)^{p-2}b_k : \IR \to [0,\infty]$, $p \ge 2$,
	\begin{equation*}
	\Var\big(\sum_{j=1}^n h(Z_j)\big) \le 4n \IE[|h(Z_0)|^2b(Z_0)]
	\end{equation*}
	where we set $b_0 \equiv 1$.
\end{lem}

\subsection{Proofs of Section \ref{sec_intro}}\label{a:intro}
\begin{proof}[Proof of Proposition \ref{prop:ass:3}]
		Let us begin with (i). For $i\in \llbracket 2\rrbracket$ we have 
		\begin{align*}
		\log(X_{1,i})-\log(V_{0,i}) &=\log\left(\Delta^{-1}\int_0^{\Delta} e^{Z_{t,i}-Z_{0,i}}dt\right)\leq \sup_{t\in [0, \Delta]} Z_{t,i}-Z_{0,i} \\
		\log(V_{0,i})-\log(X_{1,i}) &=-\log\left(\Delta^{-1}\int_0^{\Delta} e^{Z_{t,i}-Z_{0,i}}dt\right)\leq- \inf_{t\in [0, \Delta]} Z_{t,i}-Z_{0,i} 
		\end{align*}
		implying that $\E(|\log(X_{1,i})-\log(V_{0,i})|) \leq \E(\sup_{t\in [0, \Delta]} |Z_{t,i}-Z_{0,i}|)$. Next since $Z_{t,i}-Z_{0,i} = \int_0^t b_i(\bm Z_t) dt +\sum_{j=1}^2 \int_0^t a_{j,i}(\bm Z_t) d\widetilde W_t^j $ we get
		\begin{align*}
		\E(\sup_{t\in [0,\Delta]} |Z_{t,i}-Z_{0,i}|) &\leq \Delta \E(|b_i(\bm Z_0)|)+ \sum_{j=1}^2 \E\left(\sup_{t\in [0, \Delta]}\left| \int_0^t a_{i,j}(\bm Z_t) dW_t^j\right|\right)\\
		& \leq c(1+ \E(|\bm Z_0|_{\IR^2}^2)) \Delta^{1/2}
		\end{align*}
		using the Jensen inequality and the Burkholder-Davis-Gundy inequality.\\
		For (ii) we first see that for $x,y\in \IRp$ with $x<y $ holds $|\log(y)-\log(x)|=\log(y/x) =\log(1+(y-x)/x) \leq (y-x)/x=|y-x|/x$ which implies that $|\log(y)-\log(x)| \leq |y-x| /(x\wedge y)$. Here, $a\wedge b:= \min(a,b)$ for $a,b\in \IR$. We deduce that
		\begin{align*} 
		\E(|\log(X_{1,i}-\log(V_{0,i})|)^2&\leq \E(|X_{1,i}-V_{0,i}|^2) \E((V_{0,i}\wedge X_{1,i})^{-2})\\
		& \leq \E(\sup_{t\in [0,\Delta]} |V_{t,i} -V_{0,i}|^2) \E((V_{0,i}\wedge X_{1,i})^{-2}),
		\end{align*}
		Since $X_{1,i} \geq \inf_{t\in [0,\Delta]} V_{t,i}$ we get 
		$ \E((V_{0,i}\wedge X_{1,i})^{-2}))\leq \E(\sup_{t\in [0,\Delta]} (V_{t,i})^{-2}).$
		Analogously to (i) we can show that $\E(\sup_{t\in [0,\Delta]} |V_{t,i} -V_{0,i}|^2)\leq c\Delta(1+ \E(|\bm V_0|_{\IR^2}^2)).$
\end{proof}
\subsection{Proof of Section \ref{sec_svde}}
\begin{proof}[Proof of Theorem \ref{pro:volatility}]
		By a disjoint support argument and the Plancherel equality \ref{eq:Mel:plan} we have $\langle f_{\bm V}-f_{\bm V,\bm k}, f_{\bm V, \bm k}-\widehat f_{\Delta, \bm k}\rangle_{\bm x^{\ushort{1}}} = 0$ implying 
		\begin{align*}
		\|f_{\bm V}- \widehat f_{\Delta, \bm k}\|_{\bm x^{\ushort{\bm 1}}}^2 &= \|f_{\bm V}- f_{\bm V,\bm k}\|_{\bm x^{\ushort{\bm 1}}}^2+ \|f_{\bm V, \bm k}-\widehat f_{\Delta, \bm k} \|_{\bm x^{\ushort{\bm 1}}}^2.
		\end{align*}
		Let us define $f_{\Delta, \bm k}:=\E_{f_{\bm Y}}^n(\widehat f_{\Delta, \bm k})$. Then, $\E_{f_{\bm Y}}^n(\langle f_{\bm V, \bm k}- f_{\Delta, \bm k}, f_{\Delta,\bm k}-\widehat f_{\Delta, \bm k}\rangle_{\bm x^{\ushort{\bm 1}}})=0$ and
		$$\E_{f_{\bm Y}}^n(\|f_{\bm V}- \widehat f_{\Delta,\bm k}\|_{\bm x^{\ushort{\bm 1}}}^2) = \|f_{\bm V}-f_{\bm V, \bm k}\|_{\bm x^{\ushort{\bm 1}}}^2 + \|f_{\bm V, \bm k}- f_{\Delta, \bm k}\|_{\bm x^{\ushort{\bm 1}}}^2+ \E_{f_{\bm Y}}^n(\| f_{\Delta, \bm k}-\widehat f_{\Delta, \bm k}\|_{\bm x^{\ushort{\bm 1}}}^2).$$
		Following the steps of the proof of Proposition \ref{pr:consis}, we get
		$$ \E(\|f_{\Delta, \bm k}-\widehat f_{\Delta,\bm k}\|_{\bm x^{\ushort{\bm 1}}}^2) \leq \frac{\Lambda_g(\bm k)}{n}+ \frac{1}{(2\pi)^2} \int_{[-\bm k, \bm k]} \Var(\widehat{\sM}_{\bm X}(\bm t))d\bm t$$
		where $\widehat \sM_{\bm X}(\bm t):=n^{-1}\sum_{j\in \nset{n}} \bm X_j^{\ushort{i\bm t}}$. Now since $(\bm X_j)_{j\in \nset{n}}$ is a measurable transformation of $(\bm V_t)_{t\geq 0}$ which fulfills $(\bm{\mathrm A_1})$-$(\bm{\mathrm A_2})$ we have $\beta_{\bm X}(j) \leq \beta_{\bm V}(\Delta(j-1))$ and, with Lemma \ref{lem:beta_mix_var_bound}, that
		$$ \Var(\widehat \sM_{\bm X}(\bm t)) \leq \frac{4}{n} \sum_{k=1}^{\infty} \beta_{\bm V}(\Delta(k-1))\leq \frac{4}{n\Delta}\int_{\IRp} \beta_{\bm V}(s)ds.$$
		Considering the term $\|f_{\bm V,\bm k}-f_{\Delta, \bm k}\|_{\bm x^{\ushort{\bm 1}}}^2$,
		we get, since $|\bm x^{\ushort{i\bm t}}-\bm y^{\ushort{i\bm t}}| \leq  t_1 |\log(x_1)-\log(y_1)| +  t_2 |\log(x_2)-\log(y_2)|$ that
		\begin{align*}
		\|f_{\bm V, \bm k}-f_{\Delta, \bm k}\|_{\bm x^{\ushort{\bm 1}}}^2&= \frac{1}{(2\pi)^2} \int_{[-\bm k, \bm k]} |\E((\bm V_0)^{\ushort{i\bm t}}- \bm X_1^{\ushort{i\bm t}})|^2 dt\\
		& \leq \int_{[-\bm k, \bm k]} t_1^2\E(|\log(V_{0,1})-\log(X_{1,1})|)^2+ t_2^2\E(|\log(V_{0,2})-\log(X_{1,2})|)^2 d\bm t\\
		&\leq \bm k^{\ushort{\bm 3}} \E(|\log(V_{0,1})-\log(X_{1,1})|)+ |\log(V_{0,2})-\log(X_{1,2})|)^2 \leq \mathfrak c \Delta \bm k^{\ushort{\bm 3}}
		\end{align*}
		exploiting $(\bm{\mathrm A_3})$.
	\end{proof}	
	\begin{proof}[Proof of Theorem \ref{thm:pco:sv}]
		Since $\E_{f_{\bm Y}}^n(\|f_{\bm V}-\widehat f_{\Delta,\widehat{\bm k}}\|_{\bm x^{\ushort{\bm 1}}}^2) = \|f_{\bm V}-f_{\bm V, \bm K_n}\|_{\bm x^{\ushort{\bm 1}}}^2 + \E_{f_{\bm Y}}^n (\|f_{\bm V,\bm K_n }-\widehat f_{\Delta, \widehat{\bm k}}\|_{\bm x^{\ushort{\bm 1}}}^2)$ it follows
		\begin{align*} 
		\E_{f_{\bm Y}}^n(\|f_{\bm V}-\widehat f_{\Delta,\widehat{\bm k}}\|_{\bm x^{\ushort{\bm 1}}}^2) 
		\leq &\|f_{\bm V}-f_{\bm V, \bm K_n}\|_{\bm x^{\ushort{\bm 1}}}^2 + 2\|f_{\bm V,\bm K_n}-f_{\Delta,\bm K_n}\|_{\bm x^{\ushort{\bm 1}}}^2 \\
		&+2\E_{f_{\bm Y}}^n(\|f_{\Delta,\bm K_n }-\widehat f_{\Delta, \widehat{\bm k}}\|_{\bm x^{\ushort{\bm 1}}}^2).
		\end{align*}
		For the last summand we can apply the result of Theorem \ref{thm:pco:2} with $f_{\bm K_n}:=f_{\Delta,\bm K_n}$ and $\widehat f_{\bm k}:=\widehat f_{\Delta, \bm k}$ to get
		\begin{align*}
		\E_{f_{\bm Y}}^n(\|f_{\Delta,\bm K_n }-\widehat f_{\Delta, \widehat{\bm k}}\|_{\bm x^{\ushort{\bm 1}}}^2)
		\leq &3\inf_{\bm k\in \mathcal K_n} \left(\|f_{\Delta,\bm K_n}-f_{\Delta,\bm k}\|_{\bm x^{\ushort{\bm 1}}}^2+ \mathrm{pen}(\bm k)\right) + \frac{C(g) }{n}\\
		&+
		\int_{[-\bm K_n, \bm K_n]} \Var_f^n(\widehat{\sM}_{\bm X}(\bm t))d\bm t.
		\end{align*}
		Here, several summands in the bound of Theorem \ref{thm:pco:2} can be obmitted since the penality term $\mathrm{pen}(\bm k)$ has not to be estimated in the case of $\bm c=(1,1)^T$. Now since 
		\begin{align*}
		\|f_{\Delta, \bm K_n}-f_{\Delta,\bm k}\|_{\bm x^{\ushort{\bm 1}}}^2 &\leq 3(\|f_{\Delta,\bm K_n}-f_{\bm V, \bm K_n}\|_{\bm x^{\ushort{\bm 1}}}^2+ \|f_{\bm V, \bm K_n}-f_{\bm V, \bm k}\|_{\bm x^{\ushort{\bm 1}}}^2+\|f_{\bm V, \bm k}-f_{\Delta,\bm k}\|_{\bm x^{\ushort{\bm 1}}}^2) \\
		&\leq 6 \|f_{\Delta,\bm K_n}-f_{\bm V, \bm K_n}\|_{\bm x^{\ushort{\bm 1}}}^2 +  3\|f_{\bm V, \bm K_n}-f_{\bm V, \bm k}\|_{\bm x^{\ushort{\bm 1}}}^2
		\end{align*}
		we get
		\begin{align*}
		\E(\|f_{\bm V}-\widehat f_{\Delta,\widehat{\bm k}}\|_{\bm x^{\ushort{\bm 1}}}^2) \leq  C&\left( \inf_{\bm k\in \mathcal K_n}\left(\|f_{\bm V}-f_{\bm V, \bm k}\|_{\bm x^{\ushort{\bm 1}}}^2 +\mathrm{pen}(\bm k)\right) + \|f_{\Delta,\bm K_n}-f_{\bm V, \bm K_n}\|_{\bm x^{\ushort{\bm 1}}}^2  \right)\\
		& + \frac{C(g)}{n}  + 8\int_{[-\bm K_n, \bm K_n]} \Var_f^n(\widehat{\sM}_{\bm X}(\bm t))d\bm t \\
		& \leq C\inf_{\bm k\in \mathcal K_n}\left(\|f_{\bm V}-f_{\bm V, \bm k}\|_{\bm x^{\ushort{\bm 1}}}^2 +\mathrm{pen}(\bm k)\right) + C\Delta\log(n)^6 \\
		& + \frac{C(g)}{n}  +\frac{C(\beta_{\bm V}) \log^2(n)}{n\Delta}
		\end{align*}
		following the proof steps of Theorem \ref{pro:volatility} and using that $\bm K_{n}^{\ushort{\bm 1}}=K_{n,1}K_{n,2} \leq \log^2(n).$
\end{proof}
\subsection{Proof of Section \ref{sec_stat_pro}}\label{a:stat_pro}

\begin{proof}[Proof of Proposition \ref{pr:consis}]
	The Plancherel equation \eqref{eq:Mel:plan} implies $\langle f_{\bm k}, f-f_{\bm k}\rangle_{\bm x^{\ushort{2\bm c-\bm 1}}}=\langle \sM_{\bm c}[f_{\bm k}], \sM_{\bm c}[f-f_{\bm k}]\rangle_{\IR^2}=0$ since $\sM_{\bm c}[f_{\bm k}]$ and $\sM_{\bm c}[f-f_{\bm k}]$ have disjoint support. Thus
	\begin{align*}
	\E_{f_{\bm Y}}^n(\|\widehat f_{\bm k}-f\|_{\bm x^{\ushort{2\bm c-\bm 1}}}^2)&= \|f-f_{\bm k}\|_{\bm x^{\ushort{2\bm c-\bm 1}}}^2 + \E_{f_{\bm Y}}^n(\|\widehat f_{\bm k}-f_{\bm k}\|_{\bm x^{\ushort{2\bm c-\bm 1}}}^2) \\
	&=\|f-f_{\bm k}\|_{\bm x^{\ushort{2\bm c-\bm 1}}}^2 + \frac{1}{(2\pi)^2} \int_{[-\bm k, \bm k]} \frac{\Var_{f_{\bm Y}}^n(\widehat\sM_{\bm c}(\bm t))}{|\sM_{\bm c}[g](\bm t)|^2} d\bm t
	\end{align*}
	by application of the Parseval equality \eqref{eq:Mel:plan} and the Fubini-Tonelli theorem. Since for any $\bm t\in \IR^2$ $$\Var_{f_{\bm Y}}^n(\widehat \sM_{\bm c}(\bm t))= \Var^n_{f_{\bm Y}}(\widehat \sM_{\bm c}(\bm t)-\sM_{\bm c}[g](\bm t)\widehat \sM_{\bm X}(\bm t))+ |\sM_{\bm c}[g](\bm t)|^2\Var_f^n(\widehat \sM_{\bm X}(\bm t)),$$ we decompose the variance term $\E_{f_{\bm Y}}^n(\|\widehat f_{\bm k}-f_{\bm k}\|_{\bm x^{\ushort{2\bm c-\bm 1}}}^2)$ into
	\begin{align*}
	 \int_{[-\bm k, \bm k]}\frac{\Var^n_{f_{\bm Y}}(\widehat\sM_{\bm c}(\bm t))}{|\sM_{\bm c}[g](\bm t)|^2} d\bm t =&\int_{[-\bm k, \bm k]} \frac{\Var_{f_{\bm Y}}^n(\widehat\sM_{\bm c}(\bm t)-\sM_{\bm c}[g](\bm t)\widehat \sM_{\bm X}(\bm t))}{|\sM_{\bm c}[g](\bm t)|^2}d\bm t\\
	 &+\int_{[-\bm k, \bm k]}  \Var^n_{f}(\widehat\sM_{\bm X}(\bm t))d\bm t.
	\end{align*}
	Now, the equality $\E^n_{f_{\bm Y}}(\bm X_j^{\ushort{\bm c-\bm 1+i\bm t}} \bm X_{j'}^{\ushort{\bm c-\bm 1-i\bm t}}(\bm U_j^{\ushort{\bm c-\bm 1+i\bm t}}-\sM_{\bm c}[g](\bm t))(\bm U_{j'}^{\ushort{\bm c-\bm 1-i\bm t}}-\sM_{\bm c}[g](-\bm t))=\delta_{j,j'} \E_f(\bm X_1^{\ushort{2(\bm c-\bm 1)}})\Var_{g}(\bm U_j^{\ushort{\bm c-\bm 1+i\bm t}})$ implies for the first summand
	\begin{align*}
	\int_{[-\bm k, \bm k]} \frac{\Var_{f_{\bm Y}}^n(\widehat\sM_{\bm c}(\bm t)-\sM_{\bm c}[g](\bm t)\widehat \sM_{\bm X}(\bm t))}{|\sM_{\bm c}[g](\bm t)|^2}d\bm t &= \int_{[-\bm k, \bm k]} \frac{\E_f(\bm X_1^{\ushort{2(\bm c-\bm 1)}}) \Var_g(\bm U_1^{\ushort{\bm c-\bm 1+i\bm t}})}{ n|\sM_{\bm c}[g](\bm t)|^2} d\bm t \\
	&\leq \frac{(2\pi)^2\mu_Y \Lambda_g(\bm k)}{n}.
	\end{align*}
\end{proof}
\begin{proof}[Proof of Theorem \ref{thm:pco:2}]
	Let $\bm k\in \mathcal K_n$ and let us keep in mind that $[-\bm k', \bm k']=\mathrm{supp}(\mathcal M_{\bm c}[f_{\bm k'}])$, for $\bm k'\in \mathcal K_n$. Further we choose $\bm K_n\in (\IN^*)^2$ such that $K_{1,n}=K_{2,n}:=\lfloor \log(n)\rfloor$. Then for all $\bm k'\in \mathcal K_n$ holds $[-\bm k', \bm k]\subseteq [-\bm K_n, \bm K_n]$. Further, we have for any $\bm k'\in \mathcal K_n$ that $\|\widehat f_{\bm K_n}\|_{\bm x^{\ushort{2\bm c-\bm 1}}}^2- \|\widehat f_{\bm k'}\|_{\bm x^{\ushort{2\bm c-\bm 1}}}^2=\|\widehat f_{\bm K_n}-\widehat f_{\bm k'}\|_{\bm x^{\ushort{2\bm c-\bm 1}}}^2$ implying with \eqref{eq:pco:2}
	$$ \|\widehat f_{\widehat{\bm k}}-\widehat f_{\bm K_n}\|_{\bm x^{\ushort{2\bm c-\bm 1}}}^2 +\widehat{ \mathrm{pen}}(\widehat {\bm k}) \leq \|\widehat f_{\bm k}-\widehat f_{\bm K_n}\|_{\bm x^{\ushort{2\bm c-\bm 1}}}^2+ \widehat{\mathrm{pen}}(\bm k).$$
	Now for every $\bm k'\in \mathcal K_n$ we have 
	$$ \|\widehat f_{\bm k'}-f_{\bm K_n}\|_{\bm x^{\ushort{2\bm c-\bm1}}}^2 = \|\widehat f_{\bm k'}-\widehat f_{\bm K_n}\|_{\bm x^{\ushort{2\bm c-\bm 1}}}^2 + \|\widehat f_{\bm K_n}-f_{\bm K_n}\|_{\bm x^{\ushort{2\bm c-\bm 1}}}^2+2\langle \widehat f_{\bm k'}-\widehat f_{\bm K_n}, \widehat f_{\bm K_n}-f_{\bm K_n} \rangle_{\bm x^{\ushort{2\bm c-\bm 1}}}$$
	which implies that
	\begin{align} \label{eq:decomp:1:2}
	\notag\| \widehat f_{\widehat{\bm k}}-f_{\bm K_n}\|_{\bm x^{\ushort{2\bm c-\bm 1}}}^2&-\|\widehat f_{\bm k}-f_{\bm K_n}\|_{\bm x^{\ushort{2\bm c-\bm 1}}}^2 \\
	\notag&= \|\widehat f_{\widehat{\bm k}}-\widehat f_{\bm K_n}\|_{\bm x^{\ushort{2\bm c-\bm 1}}}^2-\|\widehat f_{\bm k}-\widehat f_{\bm K_n}\|_{\bm x^{\ushort{2\bm c-\bm 1}}}^2 +2\langle \widehat f_{\widehat{\bm k}}- \widehat f_{\bm k}, \widehat f_{\bm K_n}-f_{\bm K_n}\rangle_{\bm x^{\ushort{2\bm c-\bm 1}}}\\
	&\leq \widehat{\mathrm{pen}}(\bm k)- \widehat{\mathrm{pen}}(\widehat{\bm k}) + 2\langle \widehat f_{\widehat {\bm k}}-\widehat f_{\bm k}, \widehat f_{\bm K_n}-f_{\bm K_n} \rangle_{\bm x^{\ushort{2\bm c-\bm 1}}}. 
	\end{align}
	Since 
	$\langle \widehat f_{\widehat{\bm k}}-\widehat f_{\bm k}, \widehat f_{\bm K_n}-f_{\bm K_n}\rangle_{\bm x^{\ushort{2\bm c-\bm 1}}}= \|\widehat f_{\widehat{\bm k}}-f_{\widehat{\bm k}}\|_{\bm x^{\ushort{2\bm c-\bm 1}}} ^2+ \langle f_{\widehat{\bm k}}-f_{\bm k}, \widehat f_{\bm K_n}-f_{\bm K_n}\rangle_{\bm x^{\ushort{2\bm c-\bm 1}}} - \|\widehat f_{\bm k}-f_{\bm k}\|_{\bm x^{\ushort{2\bm c-\bm 1}}} ^2.
$
	we get with \eqref{eq:decomp:1:2} 
	\begin{align}\label{eq:decomp:3:2}
	\notag
	\| \widehat f_{\widehat{\bm k}}-f_{\bm K_n}\|_{\bm x^{\ushort{2\bm c-\bm 1}}}^2 &\leq \|f_{\bm k}-f_{\bm K_n}\|_{\bm x^{\ushort{2\bm c-\bm 1}}}^2 - \|\widehat f_{\bm k}-f_{\bm k}\|_{\bm x^{\ushort{2\bm c-\bm 1}}}^2+2\langle f_{\widehat{\bm k}}-f_{\bm k}, \widehat f_{\bm K_n}-f_{\bm K_n}\rangle_{\bm x^{\ushort{2\bm c-\bm 1}}} \\
	&+ \widehat{\mathrm{pen}}(\bm k)+ 2\|\widehat f_{\widehat {\bm k}}-f_{\widehat{\bm k}}\|_{\bm x^{\ushort{2\bm c-\bm 1}}}^2-\widehat{\mathrm{pen}}(\widehat{\bm k}).
	\end{align}
	Let us study the term $|2\langle f_{\widehat{\bm k}}-f_{\bm k}, \widehat f_{\bm K_n}-f_{\bm K_n}\rangle_{\bm x^{\ushort{2\bm c-\bm 1}}}|$. We remind that $\bm k'\in \mathcal K_n$
	\begin{align*}
	\|\widehat f_{\bm k'}-f_{\bm k'}\|_{\bm x^{\ushort{2\bm c-\bm 1}}}^2 = \frac{1}{(2\pi)^2}\int_{\IR^2} \mathds 1_{[-\bm k', \bm k']}(\bm t)\frac{ |\mathcal M_{\bm c}[f_{\bm Y}](\bm t)-\widehat{\mathcal M}_{\bm c}(\bm t)|^2}{|\mathcal M_{\bm c}[g](\bm t)|^2}d\bm t.
	\end{align*}
	Setting  $A^*:=[-\widehat{\bm k}, \widehat{\bm k}]\cup [-\bm k, \bm k]$ we have $\mathcal M_{\bm c}[f_{\widehat{\bm k}}-f_{ \bm k}]=\mathcal M_{\bm c}[f](\mathds 1_{[-\widehat{\bm k}, \widehat{\bm k}]}- \mathds 1_{[-\bm k, \bm k]})$ and $\mathrm{supp}(\mathcal M_{\bm c}[f_{\widehat {\bm k}}-f_{\bm k}]) \subseteq A^*\subseteq [-\bm K_n, \bm K_n]$. The Cauchy Schwarz inequality and the inequality $2ab\leq a^2+b^2$, for $a,b\in \IR$, implies
	\begin{align*}
	|2\langle f_{\widehat{\bm k}}&-f_{\bm k}, \widehat f_{\bm K_n}-f_{\bm K_n}\rangle_{\bm x^{\ushort{2\bm c-\bm 1}}}| = \frac{2}{(2\pi)^2} \left|\int_{A^*} \mathcal M_{\bm c}[f_{\widehat{\bm k}}-f_{\bm k}](\bm t) \frac{\widehat{\mathcal M}_{\bm c}(-\bm t)-\mathcal M_{\bm c}[f_{\bm Y}](-\bm  t)}{\mathcal M_{\bm c}[f_g](-\bm t)} d\bm t \right| \\
	&\leq  \frac{1}{4}\|f_{\widehat{\bm k}}-f_{\bm k}\|_{\bm x^{\ushort{2\bm c-\bm 1}}}^2 + \frac{4}{(2\pi)^2} \int_{\IR^2} \mathds 1_{A^*}(\bm t)\frac{|\mathcal M_{\bm c}[f_{\bm Y}](\bm t)- \widehat {\mathcal M}_{\bm c}(\bm t)|^2}{|\mathcal M_{\bm c}[g](\bm t)|^2} d\bm t \\
	&\leq \frac{\|f_{\widehat{\bm k}}-f_{\bm K_n}\|_{\bm x^{\ushort{2\bm c-\bm 1}}}^2}{2}+ \frac{\|f_{\bm k}-f_{\bm K_n}\|_{\bm x^{\ushort{2\bm c-\bm 1}}}}{2}+4\|\widehat f_{\bm k}-f_{\bm k}\|_{\bm x^{\ushort{2\bm c-\bm 1}}}^2 + 4\|\widehat f_{\widehat {\bm k}}-f_{\widehat{\bm k}}\|_{\bm x^{\ushort{2\bm c-\bm 1}}}^2
	\end{align*}
	using that $\mathds 1_{A^*} \leq \mathds 1_{A_k}+ \mathds 1_{A_{\widehat k}}$. Thus 
	\begin{align*}
	|2\langle f_{\widehat{\bm k}}-f_{\bm k}, \widehat f_{\bm K_n}-f_{\bm K_n}\rangle_{\bm x^{\ushort{2\bm c-\bm 1}}}| \leq  &\frac{\|f_{\bm K_n}-f_{\bm k}\|_{\bm x^{\ushort{2\bm c-\bm 1}}}^2}{2} + \frac{\|\widehat f_{\widehat {\bm k}}-f_{\bm K_n}\|_{\bm x^{\ushort{2\bm c-\bm 1}}}^2}{2} + 4\|\widehat f_{\bm k}-f_{\bm k}\|_{\bm x^{\ushort{2\bm c-\bm 1}}}^2\\
	&+ \frac{7}{2}\|\widehat f_{\widehat{\bm k}}-f_{\widehat{\bm k}}\|_{\bm x^{\ushort{2\bm c-1}}}^2
	\end{align*}
	which implies with \eqref{eq:decomp:3:2}
	\begin{align*}
	\| \widehat f_{\widehat{\bm k}}-f_{\bm K_n}\|_{\bm x^{\ushort{2\bm c-1}}}^2 \leq 3&\| f_{\bm k}-f_{\bm K_n}\|_{\bm x^{\ushort{2\bm c-1}}}^2 +6 \|\widehat f_{\bm k}-f_{\bm k}\|_{\bm x^{\ushort{2\bm c-1}}}^2  +2\widehat{\mathrm{pen}}(\bm k) \\
	&+11\|\widehat f_{\widehat{\bm k}}-f_{\widehat{\bm k}}\|_{\bm x^{\ushort{2\bm c-1}}}^2
	-2\widehat{\mathrm{pen}}(\widehat{\bm  k})
	\end{align*}
	Since $\E_{f_{\bm Y}}^n(\widehat{\mathrm{pen}}(\bm k))= \mathrm{pen}(\bm k)$ and as $\chi_0\geq 6$ we get
	Now, $
	6\E_{f_{\bm Y}}^n(\|\widehat f_{\bm k}-f_{\bm k}\|_{\bm x^{\ushort{2\bm c-\bm 1}}}^2)\leq \frac{6\mu_Y\Lambda_g(\bm k)}{n}+\int_{[-\bm k, \bm k]} \Var_f^n(\widehat{\sM}_{\bm X}(\bm t)d\bm t \leq \mathrm{pen}(\bm k)+ \int_{[-\bm k, \bm k]} \Var_f^n(\widehat{\sM}_{\bm X}(\bm t))d\bm t 
	$ 
	and
	\begin{align*} \E_{f_{\bm Y}}^n(\|\widehat f_{\widehat{\bm k}}&- f_{\bm K_n}\|_{\bm x^{\ushort{2\bm c-\bm 1}}}^2) \leq 3 \left(\|f_{\bm K_n}-f_{\bm k}\|_{\bm x^{\ushort{2\bm c-\bm 1}}}^2 +\mathrm{pen}(\bm k)\right) + \int_{[-\bm K_n, \bm K_n]} \Var_f^n(\widehat{\sM}_{\bm X}(\bm t))d\bm t\\
	&+11\E_{f_{\bm Y}}^n\left(\|\widehat f_{\widehat{\bm k}}-f_{\widehat{\bm k} }\|_{\bm x^{\ushort{2\bm c-\bm 1}}}^2-\frac{1}{12}\mathrm{pen}(\widehat{\bm k})\right)_+ +\E_{f_{\bm Y}}^n((\mathrm{pen}(\widehat{\bm k})-2\widehat{\mathrm{pen}}(\widehat{\bm k}))_+).
	\end{align*}
	The theorem follows by applying the following two Lemmas and taking the infimum over $\bm k\in \mathcal K_n$.
	\begin{lem}\label{lem:help:1}
		Under the assumptions of Theorem \ref{thm:pco:2} we get
		\begin{align*} \E_{f_{\bm Y}}^n\left(\|\widehat f_{\widehat{\bm k}}-f_{\widehat{\bm k}}\|_{\bm x^{\ushort{2\bm c-\bm 1}}}^2-\frac{1}{12}\mathrm{pen}(\widehat{\bm k})\right)_+ &\leq C(g)\left(\frac{\mu_{\bm X}}{n} + \frac{\E_f(\bm X_1^{\ushort{4(\bm c-\bm 1)}})  }{\mu_{\bm X} }\frac{\log^2(n)}{n}\right.\\ 
		&\hspace*{-1cm}+ \left.\frac{\Var_f^n(\widehat\mu_{\bm X})\log(n)^2}{\mu_{\bm X}}+ \int_{[-\bm K_n, \bm K_n]} \Var_{f}^n(\widehat{\sM}_{\bm X}(\bm t))d\bm t\right).\end{align*}
		where $C(g)>0$ is a positive constant only depending on $g$.
	\end{lem}
	\begin{lem}\label{lem:help:2}
		Under the assumptions of Theorem \ref{thm:pco:2} we get
		$$\E_{f_{\bm Y}}^n((\mathrm{pen}(\widehat{\bm k})-2\widehat{\mathrm{pen}}(\widehat{\bm k}))_+) \leq C(\chi, \mu_{\bm Y}, \E_{f_{\bm Y}}(\bm Y_1^{\ushort{4(\bm c-\bm 1)}}))\log(n)^2(n^{-1} + \Var_f ^n(\widehat\mu_{\bm X}))$$
		where $C(\chi, \mu_{\bm Y}, \E_{f_{\bm Y}}(\bm Y_1^{\ushort{4(\bm c-\bm 1)}}))>0$ is a  constant dependent on $\chi, \mu_{\bm Y},\E_{f_{\bm Y}}(\bm Y_1^{\ushort{4(\bm c-\bm 1)}})$.
	\end{lem}
\end{proof}
\begin{proof}[Proof of Lemma \ref{lem:help:1}]
	First we see that 
	$$ \E_{f_{\bm Y}}^n\left(\|\widehat f_{\widehat{\bm k}}-f_{\widehat{\bm k}}\|_{\bm x^{\ushort{2\bm c-\bm 1}}}^2-\frac{1}{12}\mathrm{pen}(\widehat{\bm k})\right)_+ \leq \E_{f_{\bm Y}}^n\left(\max_{k\in \mathcal K_n}\left(\|\widehat f_{\bm k}-f_{\bm k}\|_{\bm x^{\ushort{2\bm c-\bm 1}}}^2-\frac{1}{12}\mathrm{pen}(\bm k)\right)_+\right).$$
	Defining $B_{\bm k}:=\{h\in S_{\bm k}: \|h\|_{\bm x^{\ushort{2\bm c-\bm 1}}}=1\}$ and $\overline{\nu}_h:=\langle \widehat f_{\bm k}-f_{\bm k}, h\rangle_{\bm x^{\ushort{2\bm c-\bm 1}}}, h\in B_{\bm k}$, we have $\|\widehat f_{\bm k}-f_{\bm k}\|_{\bm x^{\ushort{2\bm c-\bm 1}}}^2 = \sup_{h\in B_{\bm k}} \overline{\nu}_h^2$. Further, we decompose $\overline{\nu}_h$ into $\bar\nu_h =\bar\nu_{h,in}+\bar\nu_{h,de}$, where
	\begin{align*}
	\bar\nu_{h,in}:= \frac{1}{n} \sum_{j\in \nset{n}} (\nu_h(\bm Y_j)- \IE_{|\bm X}(\nu_h(\bm Y_j))), \,\nu_h(\bm Y_j):= \frac{1}{4\pi^2} \int_{[-\bm k, \bm k]} \frac{\bm Y_j^{\ushort{\bm c-\bm 1+i\bm t}}}{\sM_{\bm c}[g](\bm t)} \sM_{\bm c}[h](-\bm t)d\bm t.
	\end{align*}
	and $\overline \nu_{h, de}= n^{-1} \sum_{j\in \nset{n}} \E_{|\bm X}(\nu_h(\bm Y_j))- \E_{f_Y}(\nu_h(\bm Y_j))$. 
	Thus 
	\begin{align*}\E_{f_{\bm Y}}^n\left(\|\widehat f_{\widehat{\bm k}}-f_{\widehat{\bm k}}\|_{\bm x^{\ushort{2\bm c-\bm 1}}}^2-\frac{1}{12}\mathrm{pen}(\widehat{\bm k})\right)_+ \leq 2& \E_{f_{\bm Y}}^n\left(\max_{\bm k\in \mathcal K_n}\left(\overline{\nu}_{h, in}^2-\frac{1}{24}\mathrm{pen}(\bm k)\right)_+ \right) \\
	&+ 2 \E_{f_{\bm Y}}^n\left(\max_{\bm k\in \mathcal K_n}\overline{\nu}_{h, de}^2 \right)=: I_1+I_2
	\end{align*}
	\underline{For the term $I_2$:} $\IE_{|\bm X}(\bm Y_j^{\ushort{\bm c-\bm 1+i\bm t}})- \IE_{f_{\bm Y}}(\bm Y_j^{\ushort{\bm c-\bm 1+i\bm t}}) = \sM_{\bm c}[g](\bm t)(\bm X_j^{\ushort{\bm c-\bm 1+i\bm t}}-\IE_f(\bm X_1^{\ushort{\bm c-\bm 1+i\bm t}})$ implies  for any $\bm k\in \mathcal K_n$ and any $h\in B_{\bm k}$,
	\begin{align*}
	|\overline{\nu}_{h,de}|
	= \frac{|\langle \widehat \sM_{\bm X}-\E_f^n(\widehat \sM_{\bm X}), \sM_{\bm c}[h] \rangle_{\IR^2}|}{4\pi^2} \leq \frac{\|\mathds 1_{[-\bm k, \bm k]} (\widehat\sM_{\bm X}- \E_f^n(\widehat\sM_{\bm X})\|_{\IR^2}}{2\pi},
	\end{align*}
	using the Cauchy-Schwarz inequality and $\|\sM_c[h]\|_{\IR^2}=2\pi\|h\|_{\bm x^{\ushort{2\bm c-\bm 1}}}\leq2\pi.$ Thus $$\IE_{f_{\bm Y}}^n( \max_{\bm k\in \mathcal K_n}\sup_{h\in B_{\bm k}} \bar{\nu}_{h,de}^2) \leq  \frac{1}{(2\pi)^2}\int_{[-\bm K_n, \bm K_n]} \Var_{f}^n(\widehat{\sM}_{\bm X}(\bm t))d\bm t.$$
	\underline{Next for $I_1$,} we decompose the process again to be able to apply the Talagrand inequality, \ref{tal:re}. To do so, let us define $\widetilde{\mathrm{pen}}(\bm k):=\chi\mu_{\bm U}\widehat\mu_{\bm X} \bm k^{\ushort{\bm 1}}\Lambda_g(\bm k)n^{-1}$. Then, 
	\begin{align*}
	\left(\sup_{h\in B_{\bm k}} \overline\nu_{h, in}^2-\frac{1}{24}\mathrm{pen}(\bm k)\right)_+ =\left(\sup_{h\in B_{\bm k}} \overline\nu_{h, in}^2-\frac{1}{36}\widetilde {\mathrm{pen}}(\bm k)\right)_+ +\frac{1}{24} (\frac{2}{3}\widetilde{\mathrm{pen}}(\bm k)-\mathrm{pen}(\bm k))_+.
	\end{align*}
	For the second summand, let us define $\Omega_{\bm X}:=\{|\widehat\mu_{\bm X}-\mu_{\bm X}|\leq \mu_{\bm X}/2\}$.  Then on $\Omega_{\bm X}$ we have $\widehat\mu_{\bm X} \leq 3\mu_{\bm X}/2$ and thus
	\begin{align*}
	\IE_{f_{\bm Y}}^n(\max_{\bm k\in \mathcal K_n}& (\widetilde{\mathrm{pen}}(\bm k)-\mathrm{pen}(\bm k))_+) \leq \chi (\bm K_n)^{\ushort{\bm 1}}\mu_{\bm U} \IE_f^n((\frac{2}{3}\widehat \mu_{\bm X}-\mu_{\bm X})_+\mathds 1_{\Omega_{\bm X}^c}) \\
	&= C(\chi, \mu_{\bm U}) \log^2(n)\E_f^n(|\widehat \mu_{\bm X}-\mu_{\bm X}| \mathds 1_{\Omega^c_{\bm X}}))
	\leq C(\chi, \mu_{\bm U}, \mu_{\bm X}) \log^2(n) \Var_f^n(\widehat\mu_{\bm X}),
	\end{align*}
	since $|\widehat\mu_{\bm X}- \mu_{\bm X}| \mathds 1_{\Omega_{\bm X}^c} \leq 2|\widehat\mu_{\bm X}- \mu_{\bm X}|^2 \mu_{\bm X}^{-1}$. 
	For the first summand we see
	\begin{align*}
	\IE_{f_Y}^n(\max_{\bm k\in \mathcal K_n}(\sup_{h\in B_{\bm k}} \bar{\nu}_{h,in}^2-\frac{1}{36}\widetilde{\mathrm{pen}}(\bm k))_+) =\IE_f^n(\IE_{|\bm X}(\max_{\bm k\in \mathcal K_n}(\sup_{h\in B_{\bm k}} \bar{\nu}_{h,in}^2-\frac{1}{36}\widetilde{\mathrm{pen}}(\bm k))_+)).
	\end{align*}
	Thus we start by considering the inner conditional expectation to bound the term. By the construction of $\bar\nu_{h,in}$, its summands conditioned on $\sigma(\bm X_i,i\geq 0)$ are independent but not identically distributed. We are aiming to apply the Talagrand inequality, Lemma \ref{tal:re}. We therefore split, for a sequence $(c_n)_{n\in \IN}$ specified afterwards, the process again in the following way
	\begin{align*}\bar\nu_{h,1}:= &n^{-1} \sum_{j\in \nset{n}} \nu_h(\bm Y_j) \Ii_{(0,c_n)}(\bm Y_j^{\ushort{\bm c-\bm 1}})-\IE_{|\bm X}(\nu_h(\bm Y_1) \Ii_{(0,c_n)}(\bm Y_1^{\ushort{\bm c-\bm 1}})) \\
	&\text{ and }\bar\nu_{h,2}:=  n^{-1} \sum_{j\in \nset{n}} \nu_h(\bm Y_j) \Ii_{(c_n,\infty)}(\bm Y_j^{\ushort{\bm c-\bm 1}})-\IE_{|\bm X}(\nu_h(\bm Y_1) \Ii_{(c_n,\infty)}(\bm Y_1^{\ushort{\bm c-\bm 1}}))\end{align*}
	to get
	\begin{align*}
	\IE_{|\bm X}(\max_{\bm k\in \mathcal K_n}(\sup_{h\in B_{\bm k}} |\bar{\nu}_{h,in}|^2-&\frac{1}{36}\widetilde {\mathrm{pen}}(\bm k))_+)\leq 2\IE_{|\bm X}(\max_{\bm k\in \mathcal K_n} (\sup_{h\in B_{\bm k}} |\bar{\nu}_{h,1}|^2- \frac{1}{72}\widetilde{\mathrm{pen}}(\bm k))_+)\\
	&+2\E_{|\bm X}(\max_{\bm k\in \mathcal K_n} \sup_{h\in B_{\bm k}} |\overline \nu_{h,2}|^2)
	=: M_1+M_2
	\end{align*}
	where we will now consider the two summands $M_1, M_2$ separately.\\
	To bound the $M_1$ term we will use the Talagrand inequality \ref{tal:re}. Indeed, we have
	\begin{align*}
	M_1 \leq \sum_{\bm k \leq \bm K_n}\IE_{|\bm X}(\sup_{t\in B_{\bm k}} |\bar{\nu}_{h,1}|^2-\frac{1}{72}\widetilde{\mathrm{pen}}(\bm k ))_+,
	\end{align*}
	which will be used to show the claim. We want to emphasize that we are able to apply the Talagrand inequality on the sets $B_k$ since $B_k$ has a dense countable subset and due to continuity arguments.
	Further, we see that the random variables $\nu_h(\bm Y_j)\Ii_{(0,c_n)}(\bm Y_j^{\ushort{\bm c-\bm 1}})- \IE_{|\bm X}(\nu_h(\bm Y_j)\Ii_{(0,c_n)}(\bm Y_j^{\ushort{\bm c-\bm 1}}))$, $j\in \nset{n}$, are conditioned on $\sigma(\bm X_i,i\geq 0)$, centered and independent but not identically distributed. 
	In order to apply Talagrand's inequality, we need to find the constants $\Psi, \psi, \tau$ such that
	\begin{align*}
	\sup_{h\in B_{\bm k}} \sup_{\bm y \in \IRp^2} &|\nu_h(\bm y)\Ii_{(0,c_n)}(\bm y^{\ushort{\bm c-\bm 1}})|\leq \psi; \quad \IE_{|\bm X}(\sup_{h\in B_{\bm k}} |\bar{\nu}_{h,1}|)\leq \Psi; \\
	\quad &\sup_{h\in B_{\bm k}} \frac{1}{n} \sum_{j\in \nset{n}} \Var_{|\bm X}(\nu_h(\bm Y_j)\Ii_{(0,c_n)}(\bm Y_j^{\ushort{\bm c-\bm 1}})) \leq \tau.
	\end{align*}
	We start with $\Psi^2$. Let us define $\widetilde{\mathcal M}_{\bm c}(\bm t):= n^{-1} \sum_{j\in \nset{n}} \bm Y_j^{\bm c-\bm 1+i\bm t}\Ii_{(0,c_n)}(\bm Y_j^{\bm c-\bm 1})$ as an unbiased estimator of $\sM_{\bm c}[f_{\bm Y}\Ii_{(0,c_n)}(\bm y^{\ushort{\bm c-\bm 1}})](\bm t)$ and 
	\begin{align*}
	\widetilde f_{\bm k}(x):= \frac{1}{(2\pi)^2} \int_{[-\bm k, \bm k]}\bm x^{\ushort{-\bm c-i\bm t}}\frac{\widetilde {\mathcal M}(\bm t)}{\sM_{\bm c}[g](\bm t)}d\bm t 
	\end{align*}
	where $n^{-1} \sum_{j\in \nset{n}} \nu_h(\bm Y_j) \Ii_{(0,c_n)}(\bm Y_j^{\ushort{\bm c-\bm 1}})=\langle \widetilde f_{\bm k}, h \rangle_{\bm x^{\ushort{2\bm c-\bm 1}}}.$
	Thus, we have for any $h\in B_{\bm k}$ that $\bar \nu_{h,1}^2 = \langle h, \widetilde f_{\bm k} - \IE_{|X}(\widetilde f_{\bm k}) \rangle_{\bm x^{\ushort{2\bm c-\bm 1}}}^2\leq \| h\|_{\bm x^{\ushort{2\bm c-\bm 1}}}^2 \|\widetilde f_{\bm k}- \IE_{|X}(\widetilde f_{\bm k})\|_{\bm x^{\ushort{2\bm c-\bm 1}}}^2$. Since $\|h\|_{\bm x^{\ushort{2\bm c-\bm 1}}}\leq1$, we get 
	\begin{align*}
	\IE_{|\bm X}( \sup_{h\in B_{\bm k}} \bar\nu_{h,1}^2 ) \leq \IE_{|\bm X}(\|\widetilde f_{\bm k} - \IE_{|\bm X}(\widetilde f_{\bm k})\|^2) =\frac{1}{2\pi} \int_{[-\bm k, \bm k]}\frac{\IE_{|\bm X}(|\widetilde \sM(\bm t)- \IE_{|X}(\widetilde \sM(\bm t))|^2)}{|\sM_{\bm c}[g](\bm t)|^2} d\bm t.
	\end{align*}
	Now since $\bm Y_j^{\ushort{\bm c-\bm 1+i\bm t}}\Ii_{(0,c_n)}(\bm Y_j^{\ushort{\bm c-\bm 1}})-\IE_{|\bm X}(\bm Y_j^{\ushort{\bm c-\bm 1+i\bm t}}\Ii_{(0,c_n)}(\bm Y_j^{\ushort{\bm c-\bm 1}})$ are independent conditioned on $\sigma(\bm X_i:i\geq 0)$ we obtain
	\begin{align*}
	\IE_{|\bm X}(|\widetilde \sM(\bm t)- \IE_{|\bm X}(\widetilde \sM(\bm t))|^2) \leq  \frac{1}{n^2} \sum_{j\in \nset{n}} \IE_{|\bm X}(\bm Y_j^{\ushort{2(\bm c-\bm 1)}}\Ii_{(0,c_n)}(\bm Y_j^{\ushort{\bm c-\bm 1}}))=\frac{\mu_{\bm U}}{n} \widehat\mu_{\bm X},
	\end{align*}
	which motivates the choice
	$
	\IE_{|\bm X}( \sup_{h\in B_{\bm k}} \bar\nu_{h,1}^2 ) \leq \mu_{\bm U }\widehat\mu_{\bm X} \Lambda_g(\bm k)n^{-1}=:\Psi^2.\\
	$
	Next we consider $\psi$. Let $\bm y\in \IRp^2$ and $h\in B_{k}$. Then using the Cauchy-Schwarz inequality, $|\nu_h(\bm y)\Ii_{(0,c_n)}(\bm y^{\ushort{\bm c-\bm 1}})|^2= (2\pi)^{-4} c_n^2|\int_{[-\bm k, \bm k]} \bm y^{\ushort{i\bm t}} \frac{\sM_{\bm c}[h](-\bm t)}{\sM_{\bm c}[g](\bm t)} d\bm t |^2 \leq (2\pi)^{-2} c_n^2\int_{[-\bm k, \bm k]}  |\sM_{\bm c}[g](\bm t)|^{-2}d\bm t\leq  c_n^2\Lambda_g(\bm k)=: \psi^2$ since $|\bm y^{\ushort{i\bm t}}|=1$ for all $\bm t\in \IR^2$.\\
	For $\tau$ we use the crude bound $\tau=n \Psi^2$.
	Hence, we have $\frac{n\Psi^2}{\tau}= 1$ and  $\frac{n\Psi}{\psi}=\frac{\sqrt{\sigma_U\widehat\sigma_Xn}}{c_n}$ and get 
	\begin{align*}
\IE_{|\bm X}\big(\sup_{h \in B_{\bm k} }\bar \nu_{h,1}^2 &-2(1+2\varepsilon) \mu_{\bm U} \widehat{\mu}_{\bm X} \frac{\Lambda_g(\bm k)}{n}\big)_+  \leq \frac{C}{n}\left( \widehat \mu_{\bm X}\mu_{\bm U}\Lambda_g(\bm k) \exp(-K_1\varepsilon) \right.\\
	&+\left.  \frac{\Lambda_g(\bm k) c_n^2}{n} \exp(-K_2 C_{\varepsilon} \sqrt{\varepsilon} \sqrt{\mu_{\bm U} \widehat{\mu}_{\bm X} n} c_n^{-1})\right).
	\end{align*}
	Choosing now $\varepsilon = 4\alpha_1\alpha_2 \bm k^{\ushort{\bm 1}}/K_1$ we get applying assumption \textbf{[G1]} for $\bm k\geq \bm k_g$ that $ \Lambda_g(\bm k) \exp(-K_1\varepsilon) \leq C_g\bm k^{\ushort{2 \bm \gamma}} \exp(-\bm \alpha^T \bm k)$
	which is summable over $\IN^2$. Next for $\bm k^{\ushort{\bm 1}} \geq \bm k_{g}$ we get $C_{\varepsilon}\sqrt{\varepsilon} \geq \varepsilon/2$
	and choosing $c_n:= \sqrt{n\mu_{\bm U}\widehat{\mu}_{\bm X}} 2K_2/K_1$ leading to 
	\begin{align*}
\IE_{|\bm X}\big(\sup_{h \in B_{\bm k} }\bar \nu_{h,1}^2 -2(1+2\varepsilon) \mu_{\bm U} \widehat{\mu}_{\bm X} \frac{\Lambda_g(\bm k)}{n}\big)_+  
	&\leq\frac{C_g}{n}\widehat \mu_{\bm X} \mu_{\bm U} \left( \bm k^{\ushort{2\bm\gamma}} e^{-\bm \alpha^T\bm k} + \Lambda_g(\bm k) e^{-K_1 \varepsilon}\right) \\
	&\leq \frac{C_g}{n} \widehat\mu_{\bm X} \mu_{\bm U}  \bm k^{\ushort{2\bm\gamma}} e^{-\bm \alpha^T\bm k}.
	\end{align*}
	Hence, there exists a $\chi_0>0$ such that for all $\chi>\chi_0$ holds $\frac{1}{72}\widetilde{\mathrm{pen}}(\bm k) \geq 2(1+4\alpha_1\alpha_2 \bm k^{\ushort{\bm 1}})/K_1)\mu_{\bm U} \widehat\mu_{\bm X} \Lambda_g(\bm k)n^{-1}$ implying
	\begin{align*}
	\sum_{\bm k \leq \bm K_n}\IE_{|\bm X}(\sup_{t\in B_{\bm k}} |\bar{\nu}_{h,1}|^2-\frac{1}{72}\widetilde{\mathrm{pen}}(\bm k ))_+ \leq \frac{C_g}{n}\widehat{\mu}_{\bm X} \mu_{\bm U} 
	\sum_{\bm k \leq \bm K_n} \bm k^{\ushort{2\bm \gamma}} \exp(-\bm \alpha^T \bm T) \leq \frac{C(g)\widehat\mu_{\bm X}}{n}.
	\end{align*}
	Now, we consider $M_2$. Let us define $\overline f_{\bm k}:= \widehat f_{\bm k}-\widetilde f_{\bm k}$. Then from $\overline{\nu}_{h,2}= \overline{\nu}_{h,in}- \overline\nu_{h,1}$ we deduce $\overline\nu_{h,2}^2 =\langle \overline{f}_{\bm k}-\IE_{|\bm X}(\overline f_{\bm k}), h\rangle^2_{\bm x^{\ushort{2\bm c-\bm 1}}}\leq \| \overline{f}_{\bm k}- \IE_{|\bm X}(\overline f_{\bm k}) \|_{\bm x^{\ushort{2\bm c-\bm 1}}}^2$ for any $h\in B_{\bm k}$. Further, $$\max_{\bm k\in \mathcal K_n}\| \overline{f}_{\bm k}- \IE_{|\bm X}(\overline f_{\bm k}) \|_{\bm x^{\ushort{2\bm c-\bm 1}}}^2 \leq \sum_{\bm k\in \mathcal K_n} \| \overline{f}_{\bm k}- \IE_{|\bm X}(\overline f_{\bm k}) \|_{\bm x^{2\bm c-\bm 1}}^2$$ and for each summand $\bm k\in \mathcal K_n$ we have
	\begin{align*}
	\IE_{|\bm X}(\| \overline{f}_{\bm k}- \IE_{|\bm X}(\overline f_{\bm k})) \|_{\bm x^{\ushort{2\bm c-\bm 1}}}^2) &= \frac{1}{(2\pi)^2} \int_{[-\bm k, \bm k]} \frac{\Var_{|\bm X}(\widehat{\mathcal M}(\bm t)- \widetilde{\mathcal M}(\bm t))}{ \sM_{\bm c}[g](\bm t)|^{2}} d\bm t \\
	&\leq\frac{1}{ n^2} \sum_{j=1}^n \IE_{|\bm X}(\bm Y_j^{\ushort{2(\bm c-\bm 1)}} \Ii_{(c_n, \infty)}(\bm Y_j^{\ushort{\bm c-\bm 1}})) \Lambda_g(\bm k). 
	\end{align*}
	Let us define the event $\Xi_{\bm X}:=\{\widehat\mu_{\bm X}\geq \mu_{\bm X}/2\}$. Then, we have
	\begin{align*}
	\frac{1}{ n^2} \sum_{j\in \nset{n}} \IE_{|\bm X}(\bm Y_j^{\ushort{2(\bm c-\bm 1)}} \Ii_{(c_n, \infty)} (\bm Y_j^{\ushort{(\bm c-\bm 1)}})) \Lambda_g(\bm k) \leq  \frac{C_g}{ nc_n^p} \sum_{j\in \nset{n}} \bm X_j^{\ushort{(2+p)(\bm c-\bm 1)}} \IE_g(\bm U_j^{\ushort{(2+p)(\bm c-\bm 1)}})
	\end{align*}
	where on $\Xi_{\bm X}$ we can state that $c_n^{-p} = C(g)n^{-p/2} (\widehat\mu_{\bm X})^{-p/2}  \leq C(g)\mu_{\bm X}^{-p/2} n^{-p/2} $.
	Then $p=2$ leads to $\IE_{|\bm X}(\| \overline{f}_{\bm K_n}- \IE_{|\bm X}(\overline f_{\bm K_n})) \|_{\bm x^{\ushort{2\bm c-\bm1}}}^2) \Ii_{\Xi_{\bm X}} \leq \frac{C(g)}{n^2\mu_{\bm X}} \IE_g(\bm U_1^{\ushort{4(\bm c-\bm 1)}}) \sum_{j=1}^n \bm X_j^{\ushort{4(\bm c-\bm 1)}}$. On the other hand,
	\begin{align*}
	\frac{1}{ n^2} \sum_{j=1}^n \IE_{|\bm X}(\bm Y_j^{\ushort{2(\bm c-\bm 1)}} \Ii_{(c_n, \infty)}(\bm Y_j^{\ushort{\bm c-\bm 1}})) \Lambda_g(\bm k)  \Ii_{\Xi_{\bm X}^c} \leq \frac{\mu_{\bm Y}}{2} \Ii_{\Xi_{\bm X}^c}\leq \frac{\mu_{\bm Y}}{2} \Ii_{\Omega_{\bm X}^c} .
	\end{align*}
	Using now that $|\mathcal K_n| \leq \log(n)^2$ we get 
	$$ M_2 \leq C(g)\frac{\log(n)^2}{\mu_{\bm X}} \left( \frac{\E_f(\bm X_1^{\ushort{4(\bm c-\bm 1)}})}{ n} + \Var_f^n(\widehat\mu_{\bm X}) \right).$$
	These three bounds imply
	\begin{align*}
	\E_{f_{\bm Y}}^n\left(\|\widehat f_{\widehat{\bm k}}-f_{\widehat{\bm k}}\|_{\bm x^{\ushort{2\bm c-\bm 1}}}^2-\frac{1}{12}\mathrm{pen}(\widehat{\bm k})\right)_+ \leq C(g)\big(\frac{\mu_{\bm X}}{n} + \frac{\E_f(\bm X_1^{\ushort{4(\bm c-\bm 1)}}) \log(n)^2 }{\mu_{\bm X} n} \\
	 + \frac{\Var_f^n(\widehat\mu_{\bm X})\log(n)^2}{\mu_{\bm X}}\big). 
	\end{align*}
\end{proof}

\begin{proof}[Proof of Lemma \ref{lem:help:2}]
	Let us define $\Omega:= \{|\widehat\mu_{\bm Y}-\mu_{\bm Y}|\leq \mu_{\bm Y}/2\}$. Then on $\Omega$ we have $2\widehat{\mu}_{\bm Y} \geq \mu_{\bm Y} $, respectively 
	\begin{align*} \E_{f_{\bm Y}}^n((\mathrm{pen}(\bm k)-2\widehat{\mathrm{pen}}(\bm k))_+) &= \chi \E_{f_{\bm Y}}^n\left(\widehat{\bm k}^{\ushort{\bm 1}} \frac{\Lambda_g(\widehat{\bm k})}{n}\left(\mu_{\bm Y}- 2 \widehat{\mu}_{\bm Y}\right)_+\right)  \\
	&\leq 2\chi \bm K_n^{\ushort{\bm 1}}\E_{f_{\bm Y}}^n(|\mu_{\bm Y}-  \widehat{\mu}_{\bm Y}|\mathds 1_{\Omega^c}) 
	\leq 2 \chi  \log(n)^2\frac{\Var_{f_{\bm Y}}^n(\widehat\mu_{\bm Y})}{\mu_{\bm Y}}.
	\end{align*}
	Now in analogy to the proof of \ref{pr:consis} we get
	$$ \Var_{f_{\bm Y}}^n(\widehat\mu_{\bm Y})= \frac{\E_{f_{\bm Y}}(\bm Y_1^{\ushort{4(\bm c-\bm 1)}} )}{n} + \E_g(\bm U_1^{\ushort{2(\bm c-\bm 1)}})^2 \Var_f^n(\widehat\mu_{\bm X}).$$
\end{proof}

\bibliographystyle{plain}
\bibliography{reference}

\begin{thebibliography}{10}

\bibitem{AndrewsAskeyRoy1999}
George~E. Andrews, Richard Askey, and Ranjan Roy.
\newblock {\em Special Functions}.
\newblock Encyclopedia of Mathematics and its Applications. Cambridge
  University Press, 1999.

\bibitem{asin}
Nicolas Asin and Jan Johannes.
\newblock Adaptive nonparametric estimation in the presence of dependence.
\newblock {\em Journal of Nonparametric Statistics}, 29(4):694--730, 2017.

\bibitem{BelomestnyGoldenshluger2020}
Denis Belomestny and Alexander Goldenshluger.
\newblock Nonparametric density estimation from observations with
  multiplicative measurement errors.
\newblock In {\em Annales de l'Institut Henri Poincar{\'e}, Probabilit{\'e}s et
  Statistiques}, volume~56, pages 36--67. Institut Henri Poincar{\'e}, 2020.

\bibitem{BlackScholes1973}
Fischer Black and Myron Scholes.
\newblock The pricing of options and corporate liabilities.
\newblock {\em The Journal of Political Economy}, 81(3):637--654, 1973.

\bibitem{Brenner-Miguel2021}
Sergio Brenner~Miguel.
\newblock Anisotropic spectral cut-off estimation under multiplicative
  measurement errors.
\newblock {\em Journal of Multivariate Analysis}, 190:Paper No. 104990, 18 pp.,
  2022.

\bibitem{Brenner-MiguelComteJohannes2021a}
Sergio {Brenner Miguel}, Fabienne {Comte}, and Jan {Johannes}.
\newblock {Linear functional estimation under multiplicative measurement
  errors}.
\newblock {\em arXiv e-prints}, page arXiv:2111.14920, November 2021.

\bibitem{Brenner-MiguelComteJohannes2021}
Sergio Brenner~Miguel, Fabienne Comte, and Jan Johannes.
\newblock Spectral cut-off regularisation for density estimation under
  multiplicative measurement errors.
\newblock {\em Electronic Journal of Statistics}, 15(1):3551--3573, 2021.

\bibitem{Brenner-MiguelPhandoidaen2022}
Sergio Brenner~Miguel and Nathawut Phandoidaen.
\newblock Multiplicative deconvolution in survival analysis under dependency.
\newblock {\em Statistics}, 56(2):297--328, 2022.

\bibitem{Comte2004}
Fabienne Comte.
\newblock Kernel deconvolution of stochastic volatility models.
\newblock {\em Journal of Time Series Analysis}, 25(4):563--582, 2004.

\bibitem{ComteGENON-CATALOT2006}
Fabienne Comte and Valentine Genon-Catalot.
\newblock Penalized projection estimator for volatility density.
\newblock {\em Scandinavian journal of statistics}, 33(4):875--893, 2006.

\bibitem{ComteGenon-CatalotRozenholc2010}
Fabienne Comte, Valentine Genon-Catalot, and Yves Rozenholc.
\newblock Nonparametric estimation for a stochastic volatility model.
\newblock {\em Finance and Stochastics}, 14(1):49--80, 2010.

\bibitem{ComteLacour2013}
Fabienne Comte and Claire Lacour.
\newblock Anisotropic adaptive kernel deconvolution.
\newblock {\em Annales de l'I.H.P. Probabilit\'es et statistiques},
  49(2):569--609, 2013.

\bibitem{Danielsson-Jon1998}
Jon Danielsson.
\newblock Multivariate stochastic volatility models: estimation and a
  comparison with vgarch models.
\newblock {\em Journal of Empirical Finance}, 1998.

\bibitem{Genon-CatalotJeantheauLaredo1998}
Valentine Genon-Catalot, Thierry Jeantheau, and Catherine Laredo.
\newblock Limit theorems for discretely observed stochastic volatility models.
\newblock {\em Bernoulli}, pages 283--303, 1998.

\bibitem{Genon-CatalotJeantheauLaredo1999}
Valentine Genon-Catalot, Thierry Jeantheau, and Catherine Laredo.
\newblock Parameter estimation for discretely observed stochastic volatility
  models.
\newblock {\em Bernoulli}, pages 855--872, 1999.

\bibitem{Genon-CatalotJeantheauLaredo2000}
Valentine Genon-Catalot, Thierry Jeantheau, and Catherine Lar{\'e}do.
\newblock Stochastic volatility models as hidden markov models and statistical
  applications.
\newblock {\em Bernoulli}, pages 1051--1079, 2000.

\bibitem{Genon-CatalotJeantheauLaredo2003}
Valentine Genon-Catalot, Thierry Jeantheau, and Catherine Laredo.
\newblock Conditional likelihood estimators for hidden markov models and
  stochastic volatility models.
\newblock {\em Scandinavian journal of statistics}, 30(2):297--316, 2003.

\bibitem{Gloter2000}
Arnaud Gloter.
\newblock Discrete sampling of an integrated diffusion process and parameter
  estimation of the diffusion coefficient.
\newblock {\em ESAIM: Probability and Statistics}, 4:205--227, 2000.

\bibitem{HullWhite1987}
John Hull and Alan White.
\newblock The pricing of options on assets with stochastic volatilities.
\newblock {\em The journal of finance}, 42(2):281--300, 1987.

\bibitem{KleinRio2005}
Thierry Klein and Emmanuel Rio.
\newblock Concentration around the mean for maxima of empirical processes.
\newblock {\em The Annals of Probability}, 33(3):1060--1077, 2005.

\bibitem{Meister2009}
Alexander Meister.
\newblock Density deconvolution.
\newblock In {\em Deconvolution Problems in Nonparametric Statistics}, pages
  5--105. Springer, 2009.

\bibitem{RenaultTouzi1996}
Eric Renault and Nizar Touzi.
\newblock Option hedging and implied volatilities in a stochastic volatility
  model 1.
\newblock {\em Mathematical Finance}, 6(3):279--302, 1996.

\bibitem{Schmisser2013a}
Emeline Schmisser.
\newblock Penalized nonparametric drift estimation for a multidimensional
  diffusion process.
\newblock {\em Statistics}, 47(1):61--84, 2013.

\bibitem{Talagrand1996}
Michel Talagrand.
\newblock New concentration inequalities in product spaces.
\newblock {\em Inventiones mathematicae}, 126:505--563, 1996.

\bibitem{Van-EsSpreij2011}
Bert Van~Es and Peter Spreij.
\newblock Estimation of a multivariate stochastic volatility density by kernel
  deconvolution.
\newblock {\em Journal of multivariate analysis}, 102(3):683--697, 2011.

\bibitem{Van-EsSpreijVan-Zanten2003}
Bert Van~Es, Peter Spreij, and Harry Van~Zanten.
\newblock Nonparametric volatility density estimation.
\newblock {\em Bernoulli}, 9(3):451--465, 2003.

\bibitem{viennet}
Gabrielle Viennet.
\newblock Inequalities for absolutely regular sequences: application to density
  estimation.
\newblock {\em Probab. Theory Related Fields}, 107(4):467--492, 1997.

\end{thebibliography}

\end{document}